\documentclass{amsart}
\usepackage{amssymb}
\usepackage{amsmath}
\usepackage{amsfonts}

\newtheorem{theorem}{Theorem}
\theoremstyle{plain}
\newtheorem{acknowledgement}{Acknowledgement}

\newtheorem{corollary}{Corollary}

\newtheorem{definition}{Definition}

\newtheorem{lemma}{Lemma}
\newtheorem{notation}{Notation}

\newtheorem{proposition}{Proposition}
\newtheorem{remark}{Remark}

\numberwithin{equation}{section}
\input{tcilatex}

\begin{document}
\title[On uniqueness of solutions of Navier-Stokes equations]{On uniqueness
of weak solutions of the incompressible Navier-Stokes equations in
3-dimensional case}
\author{Kamal N. Soltanov}
\address{{\small Institute of Mathematics and Mechanics National Academy of
Sciences of Azerbaijan, AZERBAIJAN; }}
\email{sultan\_kamal@hotmail.com}
\urladdr{}
\subjclass[2010]{Primary 35K55, 35K61, 35D30, 35Q30; Secondary 76D03, 76N10}
\date{}
\keywords{3D-Navier-Stokes Equations, Uniqueness, Solvability}

\begin{abstract}
In this article we study the uniqueness of the weak solution of the
incompressible Navier-Stokes Equation in the 3-dimensional case with use of
different approach. Here the uniqueness of the obtained by Leray of the weak
solution is proved in the case, when datums from spaces that are densely
contained into spaces of datums for which was proved the existence of the
weak solution. Moreover we investigate the solvability and uniqueness of the
weak solutions of problems associated with investigation of the main problem.
\end{abstract}

\maketitle

\begin{center}
\bigskip {\Large Contents}
\end{center}

1. Introduction

2. Preliminary results

3. One conditional uniqueness theorem for problem (1.1$^{1}$) - (1.3)

4. Restriction of problem (1.1$^{1}$) - (1.3)

5. Existence of Solution of Problem (3.3) - (3.5)

\qquad 5.1. A priori estamations

\qquad 5.2. Boundedness of trilinear form

\qquad 5.3. Boundedness of $u^{\prime }$

\qquad 5.4. Weakly compactness of operator $B$

\qquad 5.5. Realisation of initial condition

6. Uniqueness of Solution of Problem (3.3) - (3.5)

7. Proof of Theorem 2

8. Conclusion

9. References

\section{\label{Sec_1}Introduction}

In this article we investigate Navier-Stokes equation in the incompressible
case, i.e. we consider the following system of equations: 
\begin{equation}
\frac{\partial u_{i}}{\partial t}-\nu \Delta u_{i}+\underset{j=1}{\overset{d}{\sum }}u_{j}\frac{\partial u_{i}}{\partial x_{j}}+\frac{\partial p}{\partial x_{i}}=f_{i},\quad i=\overline{1,d},
\label{1}
\end{equation}%
\begin{equation}
\func{div}u=\underset{i=1}{\overset{d}{\sum }}\frac{\partial u_{i}}{\partial x_{i}}=0,\quad x\in \Omega \subset R^{d},t>0\quad ,
\label{2}
\end{equation}%
\begin{equation}
u\left( 0,x\right) =u_{0}\left( x\right) ,\quad x\in \Omega ;\quad u\left\vert \ _{\left( 0,T\right) \times \partial \Omega }\right. =0
\label{3}
\end{equation}%
where $\Omega \subset R^{d}$ is a bounded domain with sufficiently smooth
boundary $\partial \Omega $, $T>0$ is a positive number. As it is well known
Navier-Stokes equation describe the motion of a fluid in $R^{d}$ ($d=2$ or $%
3 $). These equations are to be solved for an unknown velocity vector $%
u(x,t)=\left\{ u_{i}(x,t)\right\} _{1}^{d}\in R^{d}$ and pressure $p(x,t)\in
R$, defined for position $x\in R^{d}$ and time $t\geq 0$, $f_{i}(x,t)$ are
the components of a given, externally applied force (e.g. gravity), $\nu $
is a positive coefficient (the viscosity), $u_{0}\left( x\right) \in R^{d}$
is a sufficiently smooth vector function (vector field).

As is well-known of \cite{Ler1} is shown (see, also, \cite{Lad1}, \cite%
{MajBer}, \cite{Con1}, \cite{Fef1}, \cite{Lio1}) that the Navier--Stokes
equations (\ref{1}), (\ref{2}), (\ref{3}) in three dimensions have a weak
solution $(u,p)$ with suitable properties. But the uniqueness of weak
solutions of the Navier--Stokes equation is not known in three space
dimensions case. Uniqueness of weak solution in two space dimensions case
were proved (\cite{LioPro}, \cite{Lio1}, see also \cite{Lad2}), and under
complementary conditions on smoothnes of the solution three dimensions case
was also studied (see, for example, \cite{Lio1}, \cite{Tem1}, \cite{Fur},
etc.).\ For the Euler equation, uniqueness of weak solutions is strikingly
false (see, \cite{Sch1}, \cite{Shn1}).

It is needed to note that the regularity of solutions in three dimensions
case were investigated and partial regularity of the suitable weak solutions
of the Navier--Stokes equation were obtained (see, \cite{Sch2}, \cite%
{CafKohNir}, \cite{Lin1}, \cite{Lio1}, \cite{Lad1}). There exist many works
which study different properties of solutions of the Navier--Stokes equation
(see, for example, \cite{Lio1}, \cite{Lad1}, \cite{Lin1}, \cite{Fef1}, \cite%
{FoiManRosTem}, \cite{FoiRosTem1}, \cite{FoiRosTem2}, \cite{FoiRosTem3}, 
\cite{GlaSveVic}, \cite{HuaWan}, \cite{PerZat}, \cite{Sol1}, \cite{Sol2}, 
\cite{Tem1}), etc.) and also different modifications of Navier--Stokes
equation (see, for example, \cite{Lad1}, \cite{Lio1}, \cite{Sol3}, etc.). \ 

It should be noted that under various complementary conditions of the type
of certain smoothness of the weak solutions different results on the
uniqueness of solution of the incompressible Navier-Stokes equation in $3D$
case earlier were obtained (see, e. g. \cite{Lad2}, \cite{Lio1}, \cite{Tem1}%
, etc.). Here we would like to note the result of article \cite{Fur} that
possesses of some proximity to the result of this article. In this article
the system of equations (1.1%
${{}^1}$%
) - (1.3), which is obtained from the incompressible Navier--Stokes system
by using of the Hopf-Leray approach is examined (that below will be
explained, e.g. as in \cite{Tem1}) in the following form 
\begin{equation*}
Nu=\overset{\bullet }{u}+\nu Au+B(u)=f,\quad \gamma _{0}u=u_{0},
\end{equation*}%
where $B(u)\equiv \underset{j=1}{\overset{3}{\sum }}u_{j}\frac{\partial u_{i}%
}{\partial x_{j}}$ and $\gamma _{0}u\equiv u\left( 0\right) $. In which the
author shows that $\left( N,\gamma _{0}\right) :Z\longrightarrow L^{2}\left(
0,T:H^{-1/2}\left( 
\Omega
\right) \right) \times H^{1/2}\left( 
\Omega
\right) $\ is the continuous operator under the condition that $%
\Omega
\subset R^{3}$ is a bounded region whose boundary $\partial 
\Omega
$ is a closed manifold of class $C^{\infty }$, where 
\begin{equation*}
Z=\left\{ \left. u\in L^{2}\left( 0,T:H^{3/2}\left( 
\Omega
\right) \right) \right\vert \ \overset{\bullet }{u}\in L^{2}\left(
0,T:H^{-1/2}\left( 
\Omega
\right) \right) \right\} .
\end{equation*}%
Moreover, here the following result is proved: if to denote by $F_{\gamma
_{0}}$ the image: $N\left( Z_{u_{0}}\right) =F_{\gamma _{0}}$ for $u_{0}\in
H^{1/2}\left( 
\Omega
\right) $ then for each $f\in F_{\gamma _{0}}$ there exists only one
solution $u\in Z$ such that $Nu=f$ and $\gamma _{0}u=u_{0}$, here $%
Z_{u_{0}}=\left\{ \left. u\in Z\right\vert \ \gamma _{0}u=u_{0}\right\} $.
In this article also the density in $L^{2}\left( 0,T:H^{-1/2}\left( 
\Omega
\right) \right) $ of the defined above set $F_{\gamma _{0}}$ in the topology
of $L^{p}\left( 0,T:H^{-l}\left( 
\Omega
\right) \right) $ is shown under certain conditions on $p,l$. Here other
interesting results for the operator $N$ relatively of the dependence of the
image of $N$ from the selected domain of definition $N$ are obtained. The
proof given in \cite{Fur} is similar to the proof of \cite{Lio1} and \cite%
{Tem1}, but the result not follows from their results.

In this article we begin with the explanation why for the study of the posed
question one must investigate the problem (1.1%
${{}^1}$%
) - (1.3). For this we use the approach Hopf-Leray (with taking into account
of the result of de Rham) for study the existence of the weak solution of
the considered problem as usually all of the above mentioned authors.

Unlike above results here we study the question on the uniqueness in the
case when the weak solution $u$ of the problem (1.1%
${{}^1}$%
) - (1.3) is contained of $\mathcal{V}\left( Q^{T}\right) $, and, as is
well-known, the following condition is sufficiently for this: the functions $%
u_{0}$ and $f$ satisfy conditions 
\begin{equation*}
u_{0}\in H\left( \Omega \right) ,\quad f\in L^{2}\left( 0,T;V^{\ast }\left(
\Omega \right) \right) .
\end{equation*}

\begin{notation}
The result obtained for the problem (1.1%
${{}^1}$%
) - (1.3) allows us to respond to the posed question, namely to prove the
uniqueness of the vector velocity $u$.
\end{notation}

So, in this article an investigation of the question on uniqueness of the
weak solutions $u$ in the sense of Hopf-Leray of the mixed problem with
Dirichlet boundary condition for the incompressible Navier-Stokes system in
the $3D$ case is given. For investigation we use an approach that is
different from usual methods which are used for investigation of the
question on the uniqueness of solution. The approach used here allows us to
receive more general result on the posed question. More precisely with use
of this approach more general uniqueness theorem of the weak solutions (of
the vector velocity $u$) of the problem obtained from mixed problem for the
incompressible Navier--Stokes equation, by using of the Hopf-Leray approach
is proved. Moreover in order to carry out the proof of the main result in
the beginning we study the auxiliary problems, more exactly we prove the
existence and uniqueness of the weak solutions of auxiliary problems.

For study of the uniqueness of solution of the problem we use the
variational formulation of the problem according to J. Leray \cite{Ler1} and
E. Hopf \cite{Hop} as above mentioned authors. As is well-known, on the
existence of solution of problem (1.1%
${{}^1}$%
) - (1.3) exist many results (see, \cite{Lio1}, \cite{Tem1} and \cite{Gal}).
We will formulate here one of these general results from the book of \cite%
{Tem1}

\begin{theorem}
(\cite{Tem1}) Let $\Omega $ be a Lipschitz open bounded set in $R^{d}$, $%
d\leq 4$. Let there be given $f$ and $u_{0}$ which

satisfy $f\in L^{2}\left( 0,T;V^{\ast }\left( \Omega \right) \right) $ and $%
u_{0}\in H\left( \Omega \right) $.\ Then there exists at least one function $%
u$ which

satisfies $u\in L^{2}\left( 0,T;V\left( \Omega \right) \right) $, $\frac{du}{%
dt}\in L^{1}\left( 0,T;V^{\ast }\left( \Omega \right) \right) $, $u\left(
0\right) =u_{0}$ and the equation 
\begin{equation}
\frac{d}{dt}\left\langle u,v\right\rangle -\left\langle \nu \Delta
u,v\right\rangle +\left\langle \underset{j=1}{\overset{d}{\sum }}u_{j}\frac{%
\partial u}{\partial x_{j}},v\right\rangle =\left\langle f,v\right\rangle
\label{1'}
\end{equation}%
for any $v\in V\left( \Omega \right) $. Moreover, $u\in L^{\infty }\left(
0,T;H\left( \Omega \right) \right) $ and $u\left( t\right) $\ is weakly
continuous from $\left[ 0,T\right] $ into $H\left( \Omega \right) $ (i. e. $%
\forall v\in H\left( \Omega \right) $, $t\longrightarrow \left\langle
u\left( t\right) ,v\right\rangle $ is a continuous scalar function, and
consequently, $\left\langle u\left( 0\right) ,v\right\rangle =\left\langle
u_{0},v\right\rangle $).
\end{theorem}

"Moreover, in the case when $d=3$ a weak solution $u$ satisfy 
\begin{equation*}
u\in V\left( Q^{T}\right) ,\quad u^{\prime }\equiv \frac{\partial u}{%
\partial t}\in L^{\frac{4}{3}}\left( 0,T;V^{\ast }(%
\Omega
)\right) ,\quad
\end{equation*}%
and also is almost everywhere equal to some continuous function from $\left[
0,T\right] $ into $H$, so that (\ref{3}) is meaningful. with use of the
obtained properties that any weak solution belong to the bounded subset of 
\begin{equation*}
\mathcal{V}\left( Q^{T}\right) \equiv V\left( Q^{T}\right) \cap
W^{1,4/3}\left( 0,T;V^{\ast }(%
\Omega
)\right)
\end{equation*}%
and satisfies the equation (\ref{1'})."

In what follows we will base on the mentioned existence theorem of the
solution of problem (1.1%
${{}^1}$%
) - (1.3) and the added notation as principal result, since we as well
investigate of the weak solution of the problem that is mentioned in Theorem
1, but by other way.

Then we can formulate the main result of this article in the following form.

\begin{theorem}
\label{Th_1}Let $\Omega \subset R^{3}$ be a domain of $Lip_{loc}$ (will be
defined below; see, Section \ref{Sec_4}), $T>0$ be a number. If given
functions $u_{0}$, $f$ satisfy of conditions $u_{0}\in H^{1/2}\left( \Omega
\right) $, $f\in L^{2}\left( 0,T;H^{1/2}\left( \Omega \right) \right) $ then
the weak solution $u\in \mathcal{V}\left( Q^{T}\right) $ of the problem (1.1%
${{}^1}$%
) - (1.3) given by the above mentioned theorem is unique.
\end{theorem}

This article is organized as follow. In Section 2 we adduce some known
results and the explanation of the relation between problems (1.1) - (1.3)
and (1.1%
${{}^1}$%
) - (1.3). This section contains some necessary technical lemmas appropriate
for the study of problem (1.1%
${{}^1}$%
) - (1.3). In Section 3 we prove one result on uniqueness of solution of
problem (1.1%
${{}^1}$%
) - (1.3) on some restriction on solution by use of the several modification
of the well-known approach. In Section 4 we by use of the new approach
transform problem (1.1%
${{}^1}$%
) - (1.3) to auxiliary problems. In Section 5 we investigate the existence
of the solution and, in Section 6 the uniqueness of solution of the
auxiliary problem. In Section 7 we prove of the main result, i.e. Theorem %
\ref{Th_1}.

\section{\label{Sec_2}Preliminary results}

In this section, we briefly recall the background material, definitions of
the appropriate spaces of Sobolev space type, deduce the necessary auxiliary
results and introduce some of the notation that is needed for the results
presented later in sections. Moreover, we recall the basic setup and results
regarding of the weak solutions of the Navier--Stokes equations used
throughout this paper. As is well known (see, e. g. \cite{Lio1}, \cite{Tem1}
and references therein) problem (\ref{1}) - (\ref{3}) possesses weak
solution in the space $\mathcal{V}\left( Q^{T}\right) \times L^{2}\left(
Q^{T}\right) $, $\mathcal{V}\left( Q^{T}\right) $ will be defined later on,
for any $u_{0i}\left( x\right) ,$ $f_{i}(x,t)$ ($i=\overline{1,3}$) which
are contained in the suitable spaces (in the case $d=3$, that we will
investigate here, essentially).

\begin{definition}
\label{D_2.1}Let $\Omega \subset R^{d}$ be a bounded Lipschitz open set and $%
Q^{T}\equiv \left( 0,T\right) \times \Omega $, $T>0$ be a number. Let $%
V\left( Q^{T}\right) $ be the space determined as 
\begin{equation*}
V\left( Q^{T}\right) \equiv L^{2}\left( 0,T;V\left( \Omega \right) \right) \cap L^{\infty }\left( 0,T;\left( H\left( \Omega \right) \right) ^{d}\right) ,%
\end{equation*}%
where $V\left( \Omega \right) $ is the closure in $\left( H_{0}^{1}\left(
\Omega \right) \right) ^{d}$\ of 
\begin{equation*}
\left\{ \varphi \left\vert \ \varphi \in \left( C_{0}^{\infty }\left( \Omega
\right) \right) ^{d},\right. \func{div}\varphi =0\right\}
\end{equation*}%
\ the dual $V\left( \Omega \right) $ determined as $V^{\ast }\left( \Omega
\right) $ and $\left( H\left( \Omega \right) \right) ^{d}$ is the closure in 
$\left( L^{2}\left( \Omega \right) \right) ^{d}$ of 
\begin{equation*}
\left\{ \varphi \left\vert \ \varphi \in \left( C_{0}^{\infty }\left( \Omega
\right) \right) ^{d},\right. \func{div}\varphi =0\right\} .
\end{equation*}%
Moreover we set also the space $\mathcal{V}\left( Q^{T}\right) \equiv
V\left( Q^{T}\right) \cap W^{1,4/3}\left( 0,T;V^{\ast }\left( \Omega \right)
\right) $.
\end{definition}

Here as is well-known $L^{2}\left( \Omega \right) $ is the Lebesgue space
and $H^{1}\left( \Omega \right) $ is the Sobolev space, that are the Hilbert
spaces and 
\begin{equation*}
H_{0}^{1}\left( \Omega \right) \equiv \left\{ v\left\vert \ v\in H^{1}\left(
\Omega \right) ,\right. v\left\vert \ _{\partial \Omega }\right. =0\right\} .
\end{equation*}%
In this case as is well-known (see, e.g. \cite{Tem1}) $H\left( \Omega
\right) $ and $V\left( \Omega \right) $ are the Hilbert spaces, also.

We assume that given functions $u_{0}$ and $f$ satisfy 
\begin{equation*}
u_{0}\in \left( H\left( \Omega \right) \right) ^{d},\quad f\in L^{2}\left(
0,T;V^{\ast }\left( \Omega \right) \right)
\end{equation*}%
where $V^{\ast }\left( \Omega \right) $ is the dual space of $V\left( \Omega
\right) $.

In order to adduce of the definition of the weak solution of the problem (1.1%
${{}^1}$%
) - (\ref{3}) we would like to note that we will investigate of the weak
solutions of problem (\ref{1}) - (\ref{3}) in the sense of J. Leray \cite%
{Ler1} by use of his approach (see, also \cite{Lio1}, \cite{Tem1}). This
approach shows that for study of the posed problem it is sufficient to
investigate of same question for the following problem by virtue of de Rham
result (see, books \cite{Lio1}, \cite{Tem1}, etc. where sufficiently clearly
explained this property of the posed problem):%
\begin{equation}
\frac{\partial u_{i}}{\partial t}-\nu \Delta u_{i}+\underset{j=1}{\overset{d}%
{\sum }}u_{j}\frac{\partial u_{i}}{\partial x_{j}}=f\left( t,x\right)
_{i},\quad i=\overline{1,d},\ \nu >0  \tag{1.1$^{1}$}
\end{equation}%
\begin{equation}
\func{div}u=\underset{i=1}{\overset{d}{\sum }}\frac{\partial u_{i}}{\partial
x_{i}}=\underset{i=1}{\overset{d}{\sum }}D_{i}u_{i}=0,\quad x\in \Omega
\subset R^{d},\ t>0,  \tag{1.2}
\end{equation}%
\begin{equation}
u\left( 0,x\right) =u_{0}\left( x\right) ,\quad x\in \Omega ;\quad
u\left\vert \ _{\left( 0,T\right) \times \partial \Omega }\right. =0. 
\tag{1.3}
\end{equation}

In order to explain that the investigation of the posed question for problem
(1.1%
${{}^1}$%
) - (\ref{3}) is sufficient for our goal we represent here some results of
the book \cite{Tem1} which have the immediate relation to this problem.

\begin{proposition}
\label{Pr_2.1}(\cite{Tem1}) Let $\Omega $ be a bounded Lipschitz open set in 
$R^{d}$ and $f=\left( f_{1},...,f_{n}\right) $, $f_{i}\in \mathcal{D}%
^{\prime }\left( \Omega \right) $, $1\leq i\leq d$. A necessary and
sufficient condition that $f=\func{grad}p$ for some $p$ in $\mathcal{D}%
^{\prime }\left( \Omega \right) $, is that $\left\langle f,v\right\rangle =0$
$\forall v\in V\left( \Omega \right) $.
\end{proposition}

\begin{proposition}
\label{Pr_2.2}(\cite{Tem1}) Let $\Omega $ be a bounded Lipschitz open set in 
$R^{d}$.

(i) If a distribution $p$ has all its first-order derivatives $D_{i}p$, $%
1\leq i\leq d$ in $L^{2}\left( \Omega \right) $, then $p\in L^{2}\left(
\Omega \right) $ and\ 
\begin{equation*}
\left\Vert p\right\Vert _{L^{2}\left( \Omega \right) /R}\leq c\left( \Omega
\right) \left\Vert \func{grad}p\right\Vert _{L^{2}\left( \Omega \right) };
\end{equation*}

(ii) If a distribution $p$ has all its first derivatives $D_{i}p$, $1\leq
i\leq d$ in $H^{-1}\left( \Omega \right) $, then $p\in L^{2}\left( \Omega
\right) $ and 
\begin{equation*}
\left\Vert p\right\Vert _{L^{2}\left( \Omega \right) /R}\leq c\left( \Omega
\right) \left\Vert \func{grad}p\right\Vert _{H^{-1}\left( \Omega \right) }.
\end{equation*}

In both cases, if $\Omega $ is any open set in $R^{d}$, then $p\in
L_{loc}^{2}\left( \Omega \right) $.
\end{proposition}

Combining these results, one can note that if $f\in H^{-1}\left( \Omega
\right) $ (or $f\in L^{2}\left( \Omega \right) $) and $(f,v)=0$, then $f=%
\func{grad}p$ with $p\in L^{2}\left( \Omega \right) $ (or $p\in H^{1}\left(
\Omega \right) $) if $\Omega $ is a Lipschitz open bounded set.

\begin{theorem}
\label{Th_1.2}(\cite{Tem1}) Let $\Omega $ be a Lipschitz open bounded set in 
$R^{d}$. Then 
\begin{equation*}
H^{\bot }=\left\{ u\in L^{2}\left( \Omega \right) :u=\func{grad}p,\ p\in
H^{1}\left( \Omega \right) \right\} ;
\end{equation*}%
\begin{equation*}
H=\left\{ u\in L^{2}\left( \Omega \right) :\func{div}u=0,\ u\left\vert
_{\partial \Omega }\ =0\right. \right\} .
\end{equation*}
\end{theorem}

\begin{lemma}
\label{L_1.1}(\cite{Tem1}) Let $V,H,V^{\ast }$ be three Hilbert spaces, each
space included in the following one $V\subset H\equiv H^{\ast }\subset
V^{\ast }$, $V^{\ast }$ being the dual of $V$ and \emph{all} the injections
are continuous. If a function $u$ belongs to $L^{2}(0,T;V)$ and its
derivative $u\prime $ belongs to $L^{2}(0,T;V^{\ast })$, then $u$ is almost
everywhere equal to a function continuous from $[0,T]$ into $H$ and we have
the following equality, which holds in the scalar distribution sense on $%
\left( 0,T\right) $: 
\begin{equation*}
\frac{d}{dt}\left\Vert u\right\Vert ^{2}=2\left\langle u\prime
,u\right\rangle .
\end{equation*}
\end{lemma}

Consequently, if one will seek of weak solution of the problem (\ref{1}) - (%
\ref{3}) by acording Hopf-Leray then one can get the following equation 
\begin{equation}
\frac{d}{dt}\left\langle u,v\right\rangle -\left\langle \nu \Delta
u,v\right\rangle +\left\langle \underset{j=1}{\overset{d}{\sum }}u_{j}\frac{%
\partial u}{\partial x_{j}},v\right\rangle =\left\langle f,v\right\rangle
-\left\langle \nabla p,v\right\rangle ,  \label{2.1}
\end{equation}%
where $v\in V(%
\Omega
)$ is arbitrary. Here if we consider of the last adding in the right side
then at illumination of above results (\ref{Pr_2.1}, \ref{Pr_2.2} and \ref%
{Th_1.2}) using integration by parts and taking into account that $v\in V(%
\Omega
)$, i.e. that $\func{div}v=0$ and $v\left\vert \ _{\left( 0,T\right) \times
\partial \Omega }\right. =0$ we get the equation 
\begin{equation}
\left\langle \nabla p,v\right\rangle \equiv \underset{\Omega }{\int }\nabla
p\cdot v\ dx=\underset{\Omega }{\int }p\func{div}v\ dx=0,\quad \forall v\in
V(%
\Omega
)  \label{2.2}
\end{equation}%
by virtue of de Rham result. Consequently taking into account (\ref{2.2}) in
(\ref{2.1}) we obtain equation (\ref{1'}) that shows why for study of the
posed question is enough to study problem (1.1%
${{}^1}$%
) - (\ref{3}).

So we can continue our investigation of problem (1.1%
${{}^1}$%
) - (\ref{3}) in the case when $d=3$.

Let $\Omega \subset R^{3}$ be a bounded domain with the boundary $\partial
\Omega $ of the Lipschitz class. We will denote by $\mathcal{H}^{1/2}\left(
\Omega \right) $ the vector space defined by 
\begin{equation*}
\mathcal{H}^{1/2}\left( \Omega \right) \equiv \left\{ w\left\vert \ w_{i}\in
H^{1/2}\left( \Omega \right) ,\right. i=1,2,3\right\} ,
\end{equation*}%
where $H^{1/2}\left( \Omega \right) $ is the Sobolev-Slobodeckij space $%
W^{1/2,2}\left( \Omega \right) $ (see, \cite{LioMag}, etc.). As is
well-known (see, e.g. \cite{LioMag}, \cite{BesIlNik} and references therein)
the trace for the function of the space $H^{1/2}\left( \Omega \right) $ is
defined, which is necessary for application of our approach to the
considered problem. We will prove the main theorem under this additional
condition that is the sufficient condition for present investigation.

\begin{definition}
\label{D_2.2}A $u\in \mathcal{V}\left( Q^{T}\right) $ is called a solution
of problem (1.1%
${{}^1}$%
) - (\ref{3}) if $u\left( t,x\right) $ satisfies the following equation 
\begin{equation*}
\frac{d}{dt}\left\langle u,v\right\rangle -\left\langle \nu \Delta
u,v\right\rangle +\left\langle \underset{j=1}{\overset{d}{\sum }}%
u_{j}D_{j}u,v\right\rangle =\left\langle f,v\right\rangle
\end{equation*}%
a. e. on $\left( 0,T\right) $ for any $v\in V\left( \Omega \right) $ and $u$
is weakly continuous from $[0,T]$ into $H$, i. e. $u\left( 0,x\right)
=u_{0}\left( x\right) $ holds.
\end{definition}

Consequently, in what follows we will use this definition together with the
standard notation that is used usually. It should be noted that in the case
when $d=3$ was proved, that the term $\underset{j=1}{\overset{3}{\sum }}%
u_{j}D_{j}u\equiv B\left( u\right) $ belong to $L^{4/3}\left( 0,T;V^{\ast
}\left( \Omega \right) \right) $ (see, e. g. the books \cite{Lio1}, \cite%
{Tem1}).

Let the posed problem have two different solutions $u,v\in \mathcal{V}\left(
Q^{T}\right) $, then within the known approach we get the following problem
for the function $w(t,x)=u(t,x)-v(t,x)$ 
\begin{equation}
\frac{1}{2}\frac{\partial }{\partial t}\left\Vert w\right\Vert _{2}^{2}+\nu \left\Vert \nabla w\right\Vert _{2}^{2}+\underset{j,k=1}{\overset{3}{\sum }}\left\langle \frac{\partial v_{k}}{\partial x_{j}}w_{k},w_{j}\right\rangle =0,
\label{2.3}
\end{equation}%
\begin{equation}
w\left( 0,x\right) \equiv w_{0}\left( x\right) =0,\quad x\in \Omega ;\quad w\left\vert \ _{\left( 0,T\right) \times \partial \Omega }\right. =0,
\label{2.4}
\end{equation}%
where $\left\langle g,h\right\rangle =\underset{i=1}{\overset{3}{\sum }}%
\underset{\Omega }{\int }g_{i}h_{i}dx$ for any $g,h\in \left( H\left( \Omega
\right) \right) ^{3}$, or $g\in V\left( \Omega \right) $ and $h\in V^{\ast
}\left( \Omega \right) $, respectively. So, for the proof of the uniqueness
of solution it is necessary to show that $w\equiv 0$ in some sense. In the
next section we will study the uniqueness by use of some modification of the
above well-known approach, which gives we only the conditional result. But
in sections 4-7 for study of the posed question we will pursue the basic
approach of this paper, therefore further in this section we consider
questions that are necessary for employing of this approach.

As our purpose is the investigation of the uniqueness of solution of problem
(1.1%
${{}^1}$%
) - (\ref{3}) therefore we will go over to the discussion of this question.
Beginning with mentioned explanations we will do some remarks about
properties of solutions of problem (1.1%
${{}^1}$%
) - (\ref{3}). As is known (\cite{Ler1}, \cite{Lad1}, \cite{Lio1}), problem
(1.1%
${{}^1}$%
) - (\ref{3}) is solvable and possesses weak solution that is contained in
the space $\mathcal{V}\left( Q^{T}\right) $ denoted in Definition \ref{D_2.1}%
. Therefore we will conduct our study under the condition that problem (1.1%
${{}^1}$%
) - (\ref{3}) have weak solutions and they belong to $\mathcal{V}\left(
Q^{T}\right) $. For the study of the uniqueness of the posed problem in the
three dimensional case we will use the ordinary approach by assuming that
problem (1.1%
${{}^1}$%
) - (\ref{3}) has, at least, two different solutions $u,v\in \mathcal{V}%
\left( Q^{T}\right) $ but by employing a different procedure we will
demonstrate that this is not possible.

Consequently, if we assume that problem (1.1%
${{}^1}$%
) - (\ref{3}) have two different solutions then they need to be different at
least on some subdomain $Q_{1}^{T}$ of $Q^{T}$. In other words there exist a
subdomain $\Omega _{1}$ of $\Omega $ and an interval $\left(
t_{1},t_{2}\right) \subseteq \left( 0,T\right] $ such that 
\begin{equation*}
Q_{1}^{T}\subseteq \left( t_{1},t_{2}\right) \times \Omega _{1}\subseteq
Q^{T}
\end{equation*}%
with $mes_{4}\left( Q_{1}^{T}\right) >0$ for which the following is true 
\begin{equation}
mes_{4}\left( \left\{ (t,x)\in Q^{T}\left\vert \ \left\vert u(t,x)-v(t,x)\right\vert \right. >0\right\} \right) =mes_{4}\left( Q_{1}^{T}\right) >0
\label{2.5}
\end{equation}%
here we denote the measure of $Q_{1}^{T}$ in $R^{4}$ as $mes_{4}\left(
Q_{1}^{T}\right) $ (Four dimensional Lebesgue measure). Whence follows, that
for the subdomain $\Omega _{1}$ takes place the inequation: $mes_{3}(\Omega
_{1})>0$.

Even though we prove the following lemmas for $d>1$, we will use them mostly
for the case $d=4$.

In the beginning we prove the following lemmas that we will use later on.

\begin{lemma}
\label{L_2.1}Let $G\subset R^{d}$ be Lebesgue measurable subset then the
following statements are equivalent:

1) $\infty >mes_{d}\left( G\right) >0;$

2) there exist a subset $I\subset R^{1}$, $mes_{1}\left( I\right) >0$ and $%
G_{\beta }\subset L_{\beta ,d-1}$, $mes_{d-1}\left( G_{\beta }\right) >0$
such that $G=\underset{\beta \in I}{\cup }G_{\beta }\cup N$, where $N$ is a
set with $mes_{d-1}\left( N\right) =0$, and $L_{\beta ,d-1}$ is the
hyperplane of $R^{d}$, with $co\dim _{d}L_{\beta ,d-1}=1$, for any $\beta
\in I$, which is generated by single vector $y_{0}\in R^{d}$ and defined in
the following form 
\begin{equation*}
L_{\beta ,d-1}\equiv \left\{ y\in R^{d}\left\vert \ \left\langle
y_{0},y\right\rangle =\beta \right. \right\} ,\quad \forall \beta \in I.
\end{equation*}
\end{lemma}

\begin{proof}
Let $mes_{d}\left( G\right) >0$ and consider the class of hyperplanes $%
L_{\gamma ,d-1}$ for which $G\cap L_{\gamma ,d-1}\neq \varnothing $ and $%
\gamma \in I_{1}$, here $I_{1}\subset R^{1}$\ be some subset. It is clear
that 
\begin{equation*}
G\equiv \underset{\gamma \in I_{1}}{\bigcup }\left\{ x\in G\cap L_{\gamma
,d-1}\left\vert \ \gamma \in I_{1}\right. \right\} .
\end{equation*}%
Then there exists a subclass of hyperplanes $\left\{ L_{\gamma
,d-1}\left\vert \ \gamma \in I_{1}\right. \right\} $ for which the
inequality $mes_{d-1}\left( G\cap L_{\gamma ,d-1}\right) >0$ is satisfied.
The number of such type hyperplanes cannot be less than countable or equal
it because $mes_{d}\left( G\right) >0$, moreover this subclass of $I_{1}$\
must possess the $R^{1}$ measure greater than $0$ since $mes_{d}\left(
G\right) >0$. Indeed, let $I_{1,0}$ be this subclass and $mes_{1}\left(
I_{1,0}\right) =0$. If we consider the set 
\begin{equation*}
\left\{ \left( \gamma ,y\right) \in I_{1,0}\times G\cap L_{\gamma
,d-1}\left\vert \ \gamma \in I_{1,0},y\in G\cap L_{\gamma ,d-1}\right.
\right\} \subset R^{d}
\end{equation*}%
where $mes_{d-1}\left( G\cap L_{\gamma ,d-1}\right) >0$ for all $\gamma \in
I_{1,0}$, but $mes_{1}\left( I_{1,0}\right) =0$, then 
\begin{equation*}
mes_{d}\left( \left\{ \left( \gamma ,y\right) \in I_{1,0}\times G\cap
L_{\gamma ,d-1}\left\vert \ \gamma \in I_{1,0}\right. \right\} \right) =0.
\end{equation*}%
On the other hand we have 
\begin{equation*}
0=mes_{d}\left( \left\{ \left( \gamma ,y\right) \in I_{1}\times G\cap
L_{\gamma ,d-1}\left\vert \ \gamma \in I_{1}\right. \right\} \right)
=mes_{d}\left( G\right)
\end{equation*}%
as $mes_{d-1}\left( G\cap L_{\gamma ,d-1}\right) =0$ for all $\gamma \in
I_{1}-I_{1,0}$. But this contradicts the condition $mes_{d}\left( G\right)
>0 $. Consequently, the statement 2 holds.

Let the statement 2 holds. It is clear that the class of hyperplanes $%
L_{\beta ,d-1}$ defined by such way are paralell and also we can define the
class of subsets of $G$ as its cross-section with hyperplanes, i.e. in the
form: $G_{\beta }\equiv G\cap L_{\beta ,d-1}$, $\beta \in I$. Then $G_{\beta
}\neq \varnothing $ and we can write $G_{\beta }\equiv G\cap L_{\beta ,d-1}$%
, $\beta \in I$, moreover $G\equiv \underset{\beta \in I}{\bigcup }\left\{
x\in G\cap L_{\beta ,d-1}\left\vert \ \beta \in I\right. \right\} \cup N$.
Whence we get 
\begin{equation*}
G\equiv \left\{ \left( \beta ,x\right) \in I\times G\cap L_{\beta
,d-1}\left\vert \ \beta \in I,x\in G\cap L_{\beta ,d-1}\right. \right\} \cup
N.
\end{equation*}

Consequently $mes_{d}\left( G\right) >0$ by virtue of conditions $%
mes_{1}\left( I\right) >0$ and $\ $

$mes_{d-1}\left( G_{\beta }\right) >0$ for any $\beta \in I$.
\end{proof}

From Lemma \ref{L_2.1} it follows that for the study of the measure of some
subset $%
\Omega
\subseteq R^{d}$ it is enough to study its foliations by a class of suitable
hyperplanes.

\begin{lemma}
\label{L_2.2}Let problem (1.1%
${{}^1}$%
) - (\ref{3}) has, at least, two different solutions $u,v$ that are
contained in $\mathcal{V}\left( Q^{T}\right) $ and assume that $%
Q_{1}^{T}\subseteq Q^{T}$ is one of a subdomain of $Q^{T}$ where $u$ and $v$
are different. Then there exists, at least, one class of parallel and
different hyperplanes $L_{\alpha }$, $\alpha \in I\subseteq \left( \alpha
_{1},\alpha _{2}\right) \subset R^{1}$ ($\alpha _{2}>\alpha _{1}$)\ with $%
co\dim _{R^{3}}L_{\alpha }=1$ such, that $u\neq v$ on $Q_{L_{\alpha
}}^{T}\equiv \left[ \left( 0,T\right) \times \left( \Omega \cap L_{\alpha
}\right) \right] \cap Q_{1}^{T}$, and vice versa, here $mes_{1}\left(
I\right) >0$ and $L_{\alpha }$ are hyperplanes which are defined as follows:
there is vector $x_{0}\in S_{1}^{R^{3}}\left( 0\right) $ such that 
\begin{equation*}
L_{\alpha }\equiv \left\{ x\in R^{3}\left\vert \ \left\langle
x_{0},x\right\rangle =\alpha ,\right. \ \forall \alpha \in I\right\} .
\end{equation*}
\end{lemma}

\begin{proof}
Let problem (1.1%
${{}^1}$%
) - (\ref{3}) have two different solutions $u,v\in \mathcal{V}\left(
Q^{T}\right) $ then there exist a subdomain of $Q^{T}$ on which these
solutions are different. Then there are $t_{1},t_{2}>0$ such, that for any $%
t\in J\subseteq \left[ t_{1},t_{2}\right] \subseteq \left[ 0,T\right) $ the
following holds 
\begin{equation}
mes_{R^{3}}\left( \left\{ x\in \Omega \left\vert \ \left\vert u\left( t,x\right) -v\left( t,x\right) \right\vert >0\right. \right\} \right) >0
\label{2.6}
\end{equation}%
where $mes_{1}\left( J\right) >0$ by the virtue of the codition 
\begin{equation*}
mes_{4}\left( \left\{ (t,x)\in Q^{T}\left\vert \ \left\vert
u(t,x)-v(t,x)\right\vert \right. >0\right\} \right) >0
\end{equation*}%
and of Lemma \ref{L_2.1}. Hence follows, that there exist, at least, one
class of parallel hyperplanes $L_{\alpha }$, $\alpha \in I\subseteq \left(
\alpha _{1},\alpha _{2}\right) \subset R^{1}$ with $co\dim _{R^{3}}L_{\alpha
}=1$ such that 
\begin{equation}
mes_{R^{2}}\left\{ x\in \Omega \cap L_{\alpha }\left\vert \ \left\vert u\left( t,x\right) -v\left( t,x\right) \right\vert >0\right. \right\} >0,\ \forall \alpha \in I
\label{2.7}
\end{equation}%
for $\forall t\in J$, where the subset $I$ is such that $I\subseteq \left(
\alpha _{1},\alpha _{2}\right) \subset R^{1}$ with $mes_{1}\left( I\right)
>0 $, $mes_{1}\left( J\right) >0$ and (\ref{2.7}) holds, by virtue of (\ref%
{2.6}). This proves the "if" part of Lemma.

Now consider the converse assertion. Let there exist a class of hyperplanes $%
L_{\alpha }$, $\alpha \in I_{1}\subseteq \left( \alpha _{1},\alpha
_{2}\right) \subset R^{1}$ with $co\dim _{R^{3}}L_{\alpha }=1$ that fulfills
the condition of Lemma and $I_{1}$\ satisfies the same condition $I$. Then
there exist, at least, one subset $J_{1}$ of $\left[ 0,T\right) $ such, that 
$mes_{1}\left( J_{1}\right) >0$ and the inequality $u\left( t,x\right) \neq
v\left( t,x\right) $ on $Q_{2}^{T}$ with $mes_{4}\left( Q_{2}^{T}\right) >0$
defined as $Q_{2}^{T}\equiv J_{1}\times U_{L}$ takes place, where 
\begin{equation}
U_{L}\equiv \underset{\alpha \in I_{1}}{\bigcup }\left\{ x\in \Omega \cap L_{\alpha }\left\vert \ u\left( t,x\right) \neq v\left( t,x\right) \right. \right\} \subset \Omega ,\ t\in J_{1}
\label{2.8}
\end{equation}%
for which the inequality $mes_{R^{3}}\left( U_{L}\right) >0$ is satisfied by
the condition and of Lemma \ref{L_2.1}.

So we get 
\begin{equation*}
u\left( t,x\right) \neq v\left( t,x\right) \quad \text{on \ }Q_{2}^{T}\equiv
J_{1}\times U_{L},\quad mes_{4}\left( Q_{2}^{T}\right) >0.
\end{equation*}%
Thus the fact that $u\left( t,x\right) $ and $v\left( t,x\right) $ are
different functions in $\mathcal{V}\left( Q^{T}\right) $ follows.
\end{proof}

It is not difficult to see that result of Lemma \ref{L_2.2} is indepandent
of assumption: $Q_{1}^{T}\subset Q^{T}$ or $Q_{1}^{T}=Q^{T}$.

May be one can prove more general lemmas of such type with the use of
regularity properties of weak solutions of this problem (see, \cite{Sch1}, 
\cite{CafKohNir}, \cite{Lin1}, etc.).

\section{\label{Sec_3}One conditional uniqueness theorem for problem (1.1$%
^{1}$) - (1.3)}

In the beginning we would like to show a what result one can receive
relatively of the uniqueness of solution of the problem without of the
complementary conditions. In other words we will show a what result one can
obtain if to apply well-known approach for the investigation of the
uniqueness of solution of problem (1.1%
${{}^1}$%
) - (\ref{3}).

Let the posed problem have two different solutions $u,v\in \mathcal{V}\left(
Q^{T}\right) $, then within the known approach we get the following problem
for the vector function $w(t,x)=u(t,x)-v(t,x)$ 
\begin{equation}
\frac{1}{2}\frac{\partial }{\partial t}\left\Vert w\right\Vert _{2}^{2}+\nu \left\Vert \nabla w\right\Vert _{2}^{2}+\underset{j,k=1}{\overset{3}{\sum }}\left\langle \frac{\partial v_{k}}{\partial x_{j}}w_{k},w_{j}\right\rangle =0,
\label{2.9}
\end{equation}%
\begin{equation}
w\left( 0,x\right) \equiv w_{0}\left( x\right) =0,\quad x\in \Omega ;\quad w\left\vert \ _{\left( 0,T\right) \times \partial \Omega }\right. =0,
\label{2.10}
\end{equation}%
Here we will show a result when the solution of problem (\ref{2.9})-(\ref%
{2.10}) only is zero, with use some approach that is based on the nature of
the nonlinearity of this problem. Consequently, for examination of the
problem (\ref{2.9})-(\ref{2.10}) in the beginning we need to study the
following quadratic form (\cite{Gan}) 
\begin{equation*}
\underset{j,k=1}{\overset{d}{\sum }}\left( \frac{\partial v_{k}}{\partial
x_{j}}w_{k}w_{j}\right) \left( t,x\right) \Longrightarrow F\left( t,x\right)
\equiv \underset{j,k=1}{\overset{d}{\sum }}\left( a_{jk}w_{k}w_{j}\right)
\left( t,x\right) \Longrightarrow
\end{equation*}%
\begin{eqnarray*}
F\left( t,x\right) &\equiv &\underset{j=1}{\overset{d}{\sum }}\left( a_{j}%
\overline{w}_{j}^{2}\right) \left( t,x\right) ,\quad a_{j}\left( t,x\right)
\equiv G_{j}\left( D_{1}v_{1},...,D_{1}v_{d},...,D_{d}v_{d}\right) \\
\text{here \ \ \ }D_{i}v_{k} &\equiv &\frac{\partial v_{k}}{\partial x_{i}}%
,\quad i,k=\overline{1,d}
\end{eqnarray*}%
i.e. the behavior of the surface generated by the quadratic polynomial
function $F\left( t,x\right) $, at the varabale $w_{k},\ k=\overline{1,d}$,
depende of the accelerations of the flow on the different directions.

Assume $d=3$ then we have 
\begin{equation*}
a_{1}=D_{1}v_{1};\ a_{2}=D_{2}v_{2}-\frac{\left(
D_{1}v_{2}+D_{2}v_{1}\right) ^{2}}{4a_{1}};\ a_{3}=\frac{\det \left\Vert
D_{i}v_{k}\right\Vert _{i,k=1}^{3}}{\det \left\Vert D_{i}v_{k}\right\Vert
_{i,k=1}^{2}},
\end{equation*}%
therefore, 
\begin{equation}
F\left( t,x\right) \equiv \underset{j,k=1}{\overset{3}{\sum }}\left(
a_{jk}w_{k}w_{j}\right) \left( t,x\right) \equiv \underset{j=1}{\overset{3}{%
\sum }}a_{j}\left( t,x\right) \cdot \overline{w}_{j}^{2}\left( t,x\right) =
\label{2.11}
\end{equation}%
for any $\left( t,x\right) \in Q^{T}\equiv \left( 0,T\right) \times \Omega $%
, here

\begin{center}
$\left\Vert D_{i}v_{k}\right\Vert _{i,k=1}^{3}\equiv \left\Vert 
\begin{array}{ccc}
D_{1}v_{1} & \frac{1}{2}\left( D_{1}v_{2}+D_{2}v_{1}\right) & \frac{1}{2}%
\left( D_{1}v_{3}+D_{3}v_{1}\right) \\ 
\frac{1}{2}\left( D_{1}v_{2}+D_{2}v_{1}\right) & D_{2}v_{2} & \frac{1}{2}%
\left( D_{2}v_{3}+D_{3}v_{2}\right) \\ 
\frac{1}{2}\left( D_{1}v_{3}+D_{3}v_{1}\right) & \frac{1}{2}\left(
D_{2}v_{3}+D_{3}v_{2}\right) & D_{3}v_{3}%
\end{array}%
\right\Vert $
\end{center}

and

\begin{center}
$\left\Vert D_{i}v_{k}\right\Vert _{i,k=1}^{2}\equiv \left\Vert 
\begin{array}{cc}
D_{1}v_{1} & \frac{1}{2}\left( D_{1}v_{2}+D_{2}v_{1}\right) \\ 
\frac{1}{2}\left( D_{1}v_{2}+D_{2}v_{1}\right) & D_{2}v_{2}%
\end{array}%
\right\Vert .$
\end{center}

Thus we have 
\begin{equation*}
F\left( t,x\right) \equiv \frac{1}{D_{1}v_{1}}\left[ 2D_{1}v_{1}w_{1}+\left(
D_{1}v_{2}+D_{2}v_{1}\right) w_{2}+\left( D_{1}v_{3}+D_{3}v_{1}\right) w_{3}%
\right] ^{2}+
\end{equation*}%
\begin{equation*}
\frac{1}{\left( 4D_{1}v_{1}\right) ^{2}}\left( 4D_{1}v_{1}D_{2}v_{2}-\left(
D_{1}v_{2}+D_{2}v_{1}\right) ^{2}\right) \times
\end{equation*}%
\begin{equation*}
\left[ \left( 4D_{1}v_{1}D_{2}v_{2}-\left( D_{1}v_{2}+D_{2}v_{1}\right)
^{2}\right) w_{2}\right. +
\end{equation*}%
\begin{equation*}
\left. \left( 2D_{1}v_{1}\left( D_{2}v_{3}+D_{3}v_{2}\right) -\left(
D_{1}v_{2}+D_{2}v_{1}\right) \left( D_{1}v_{3}+D_{3}v_{1}\right) \right)
w_{3}\right] ^{2}+
\end{equation*}%
\begin{equation*}
\frac{1}{4}\left[ 4D_{1}v_{1}D_{2}v_{2}D_{3}v_{3}+\left(
D_{1}v_{2}+D_{2}v_{1}\right) \left( D_{1}v_{3}+D_{3}v_{1}\right) \left(
D_{2}v_{3}+D_{3}v_{2}\right) \right. -
\end{equation*}%
\begin{equation*}
\left. D_{1}v_{1}\left( D_{2}v_{3}+D_{3}v_{2}\right) ^{2}-D_{2}v_{2}\left(
D_{1}v_{3}+D_{3}v_{1}\right) ^{2}-D_{3}v_{3}\left(
D_{1}v_{2}+D_{2}v_{1}\right) ^{2}\right] w_{3}^{2}.
\end{equation*}

Hence, if one take account (\ref{2.11}) in the equation (\ref{2.9}) then we
get 
\begin{equation*}
\frac{1}{2}\frac{\partial }{\partial t}\left\Vert w\right\Vert _{2}^{2}+\nu
\left\Vert \nabla w\right\Vert _{2}^{2}+\underset{j=1}{\overset{3}{\sum }}%
\left\langle a_{j}\overline{w}_{j},\overline{w}_{j}\right\rangle =0,
\end{equation*}%
and consequently, 
\begin{equation}
\frac{1}{2}\frac{\partial }{\partial t}\left\Vert w\right\Vert _{2}^{2}=-\nu
\left\Vert \nabla w\right\Vert _{2}^{2}-\underset{j=1}{\overset{3}{\sum }}%
\left\langle a_{j}\overline{w}_{j},\overline{w}_{j}\right\rangle ,\quad
\left\Vert w_{0}\right\Vert _{2}=0.  \label{2.12}
\end{equation}

This shows that if the quadratic form function $F\left(
t,x,w_{i},w_{j}\right) =\underset{j=1}{\overset{3}{\sum }}a_{j}\left(
t,x\right) \cdot \overline{w}_{j}^{2}\left( t,x\right) $ for a.e. $\left(
t,x\right) \in Q^{T}$ is the quadratic polynomial that describe an ellipsoid
in $R^{3}$ at the variables $\left( \overline{w}_{i},\overline{w}_{j}\right) 
$, i.e. $a_{j}\left( t,x\right) \geq 0$, then the posed problem have unique
solution. It is needed to note that images of functions $a_{j}\left(
t,x\right) $ are contained a.e. in the bounded subset of same space where
are contained images of functions $D_{i}v_{k}$. So is remains to investigate
the cases when the above condition not is fulfilled.

Here the following variants are possible:

1. Let $F\left( t,x,w_{i},w_{j}\right) $ describe some other surface in $%
R^{3}$, but 
\begin{equation*}
\underset{j=1}{\overset{3}{\sum }}\left\langle a_{j}\overline{w}_{j},%
\overline{w}_{j}\right\rangle \equiv \underset{j=1}{\overset{3}{\sum }}{}%
\underset{\Omega }{\int }a_{j}\overline{w}_{j}^{2}dx\geq 0.
\end{equation*}

In this case we can conclude that problem have unique solution (and a
solution is stable).

2. Let $\underset{j=1}{\overset{3}{\sum }}{}\underset{\Omega }{\int }a_{j}%
\overline{w}_{j}^{2}dx<0$ then we will study $F\left( t,x\right) \equiv 
\underset{j,k=1}{\overset{3}{\sum }}\left( D_{i}v_{k}w_{k}w_{j}\right)
\left( t,x\right) $.

In this case the problem (\ref{2.12}) is possible to investigate by
following way. For the multilinear form in the equation it is necessary to
derive suitable estimations. Thus

\begin{equation*}
\frac{1}{2}\frac{\partial }{\partial t}\left\Vert w\right\Vert _{2}^{2}=-\nu
\left\Vert \nabla w\right\Vert _{2}^{2}+\left\vert \underset{j,k=1}{\overset{%
3}{\sum }}\left\langle D_{i}v_{k}w_{k},w_{j}\right\rangle \right\vert \leq
\end{equation*}%
\begin{equation}
-\underset{j=1}{\overset{3}{\sum }}{}\underset{\Omega }{\int }\nu \left\vert
\nabla w_{j}\left( t,x\right) \right\vert ^{2}dx+\underset{j,k=1}{\overset{3}%
{\sum }}{}\underset{\Omega }{\int }\left\vert \left(
D_{i}v_{k}w_{k}w_{j}\right) \left( t,x\right) \right\vert dx  \label{2.13}
\end{equation}%
as in this case the second adding in the right part is negative. So, in
order that to continue the inequation (\ref{2.13}), we will use of the
corresponding estimations. Here using H\={o}lder inequality to multilinear
term we obtain the estimation \footnote{%
It is known that (\cite{Lad1}, \cite{Lio1}) $\left\vert \left\langle
u_{k}D_{i}v_{j},w_{l}\right\rangle \right\vert \leq \left\Vert
u_{k}\right\Vert _{q}\left\Vert D_{i}v_{j}\right\Vert _{2}\left\Vert
w_{l}\right\Vert _{n},\quad n\geq 3;$%
\par
$\left\Vert v_{j}\right\Vert _{4}\leq C\left( mes\ \Omega \right) \left\Vert
Dv_{j}\right\Vert _{2}^{\frac{1}{2}}\left\Vert v_{j}\right\Vert _{2}^{\frac{1%
}{2}},\quad n=2$} 
\begin{equation*}
\left\vert \left\langle D_{i}v_{j}w_{i},w_{j}\right\rangle \right\vert \leq
\left\Vert D_{i}v_{j}\right\Vert _{2}\left\Vert w_{i}\right\Vert
_{p_{1}}\left\Vert w_{j}\right\Vert _{p_{2}},
\end{equation*}%
where $p_{1}^{-1}+p_{2}^{-1}=2^{-1}$, for us sufficiently to choose, $%
p_{1}=p_{2}=4$. Consequently, for $F\left( t,x,w_{i},w_{j}\right) $ takes
place the estimation 
\begin{equation*}
\underset{\Omega }{\int }\left\vert F\left( t,x\right) \right\vert dx\leq 
\underset{i,j=1}{\overset{3}{\sum }}\left\Vert D_{j}v_{i}\right\Vert
_{2}\left\Vert w_{i}\right\Vert _{4}\left\Vert w_{j}\right\Vert _{4}.
\end{equation*}

Here we will use the known Gagliardo-Nirenberg-Sobolev inequation, that can
be formulated as follows (see, for example, \cite{BesIlNik}) 
\begin{equation}
\left\Vert u\right\Vert _{p_{0},s}\leq C\left( p_{0},p_{1},p_{2},s,m\right)
\left( \underset{\left\vert \alpha \right\vert =m}{\sum }\left\Vert
D^{\alpha }u\right\Vert _{p_{1}}\right) ^{\sigma }\cdot \left\Vert
u\right\Vert _{p_{2}}^{1-\sigma },  \label{2.14}
\end{equation}%
\textit{inequation holds if }$1\leq p_{1},p_{2},p_{0}\leq \infty ,$\textit{\ 
}$0\leq s<m,$\textit{\ where }$C\left( p_{0},p_{1},p_{2},s,m\right) >0$%
\textit{\ is constant, \ }%
\begin{equation*}
\frac{d}{p_{0}}-s=\sigma \left( \frac{d}{p_{1}}-m\right) +\left( 1-\sigma
\right) \frac{d}{p_{2}},\quad \frac{s}{m}\leq \sigma \leq 1
\end{equation*}%
\textit{with the following exclusions:}

\textit{a) if }$s=0,s<\frac{d}{p_{1}},p_{2}=\infty $\textit{\ then (\ref%
{2.14}) holds under complementary condition: or }$\underset{x\longrightarrow
\infty }{\lim }u\left( x\right) =0$\textit{, or }$u\in L^{q}$\textit{\ for
some }$q>0;$

\textit{b) if }$1\leq p_{1}<\infty ,m-s-\frac{n}{p_{1}}=0,p_{0}=\infty $%
\textit{\ then (\ref{2.14}) not holds in the case }$\sigma =1$\textit{.}

{}Hence use G-N-S inequation we get 
\begin{equation*}
\left\Vert w_{j}\right\Vert _{4}\leq c\left\Vert w_{j}\right\Vert
_{2}^{1-\sigma }\left\Vert \nabla w_{j}\right\Vert _{2}^{\sigma },\quad
\sigma =\frac{3}{4},
\end{equation*}%
here $c\equiv C\left( 4,2,2,0,1\right) $ for this case, or 
\begin{equation*}
\left\Vert w_{j}\right\Vert _{4}\leq c\left\Vert w_{j}\right\Vert _{2}^{%
\frac{1}{4}}\left\Vert \nabla w_{j}\right\Vert _{2}^{\frac{3}{4}%
}\Longrightarrow \left\Vert w_{j}\right\Vert _{4}^{2}\leq c^{2}\left\Vert
w_{j}\right\Vert _{2}^{\frac{1}{2}}\left\Vert \nabla w_{j}\right\Vert _{2}^{%
\frac{3}{2}}.
\end{equation*}%
Thus 
\begin{equation*}
\underset{\Omega }{\int }\left\vert F\left( t,x\right) \right\vert dx\leq
c^{2}\underset{i,j=1}{\overset{3}{\sum }}\left\Vert D_{j}v_{i}\right\Vert
_{2}\left\Vert w_{i}\right\Vert _{2}^{\frac{1}{4}}\left\Vert \nabla
w_{i}\right\Vert _{2}^{\frac{3}{4}}\left\Vert w_{j}\right\Vert _{2}^{\frac{1%
}{4}}\left\Vert \nabla w_{j}\right\Vert _{2}^{\frac{3}{4}}
\end{equation*}%
if one take into account the above estimation in (\ref{2.14}) then 
\begin{equation*}
\frac{1}{2}\frac{\partial }{\partial t}\left\Vert w\left( t\right)
\right\Vert _{2}^{2}\leq -\underset{j=1}{\overset{3}{\sum }}{}\nu \left\Vert
\nabla w_{j}\left( t\right) \right\Vert _{2}^{2}+c^{2}\underset{i,j=1}{%
\overset{3}{\sum }}{}\left\Vert D_{j}v_{i}\left( t\right) \right\Vert
_{2}\left\Vert w_{i}\left( t\right) \right\Vert _{2}^{\frac{1}{2}}\left\Vert
\nabla w_{i}\left( t\right) \right\Vert _{2}^{\frac{3}{2}}
\end{equation*}%
\begin{equation*}
\leq -\underset{j=1}{\overset{3}{\sum }}\left\Vert \nabla w_{j}\left(
t\right) \right\Vert _{2}^{\frac{3}{2}}\left[ \nu \left\Vert \nabla
w_{j}\left( t\right) \right\Vert _{2}^{\frac{1}{2}}-c^{2}\underset{i=1}{%
\overset{3}{\sum }}\left\Vert D_{i}v_{j}\left( t\right) \right\Vert
_{2}\left\Vert w_{j}\left( t\right) \right\Vert _{2}^{\frac{1}{2}}\right]
\end{equation*}%
\begin{equation*}
\leq -\underset{j=1}{\overset{n}{\sum }}\left\Vert \nabla w_{j}\left(
t\right) \right\Vert _{2}^{\frac{3}{2}}\left[ \nu \lambda _{1}^{\frac{1}{4}%
}-c^{2}\underset{i=1}{\overset{n}{\sum }}\left\Vert D_{i}v_{j}\left(
t\right) \right\Vert _{2}\right] \left\Vert w_{j}\left( t\right) \right\Vert
_{2}^{\frac{1}{2}}.
\end{equation*}%
From here is easy follows, that if $\nu \lambda _{1}^{\frac{1}{4}}\geq c^{2}%
\underset{i=1}{\overset{3}{\sum }}\left\Vert D_{i}v_{j}\left( t\right)
\right\Vert _{2}$ then the considered problem (1.1$^{1}$)-(\ref{3}) has only
unique solution (and a solution is stable). Thus is proved

\begin{theorem}
\label{Th_3.1}Let $\Omega \in R^{3}$ be a bounded domain of Lipschitz class
and $\left( u_{0},f\right) \in \left( H\left( \Omega \right) \right)
^{3}\times L^{2}\left( 0,T;V^{\ast }\left( \Omega \right) \right) $ then as
is well-known weak solution $u\left( t,x\right) $ of problem (1.1$^{1}$)-(%
\ref{3}) exists and $u\in \mathcal{V}\left( Q^{T}\right) $. Then if $%
\underset{\Omega }{\int }\left\vert F\left( t,x\right) \right\vert dx\geq 0$%
, or $\underset{\Omega }{\int }\left\vert F\left( t,x\right) \right\vert
dx<0 $ and $\nu \lambda _{1}^{\frac{1}{4}}\geq c^{2}\underset{i=1}{\overset{3%
}{\sum }}\left\Vert D_{i}u_{j}\left( t\right) \right\Vert _{2}$ are
fulfilled then weak solution $u\left( t,x\right) $ is unique.
\end{theorem}

\section{\label{Sec_4}Restriction of problem (1.1$^{1}$) - (1.3)}

From Lemma \ref{L_2.2} it follows that for the investigation of the posed
question it is enough to investigate this problem on the suitable
cross-sections of the domain $Q^{T}\equiv \left( 0,T\right) \times \Omega $.
We introduce the following concept.

\begin{definition}
\label{D_3.1}\textit{A bounded domain }$%
\Omega
\subset R%
{{}^3}%
$\textit{\ with the boundary }$\partial 
\Omega
$\textit{\ is class }$Lip_{loc}$\textit{\ iff }$\partial 
\Omega
$\textit{\ is a locally Lipschitz hypersurface. This means that in a
neighbourhood of any point }$x\in \partial 
\Omega
$\textit{\ , }$\partial 
\Omega
$\textit{\ admits a }representation as a hypersurface $y_{3}=\psi \left(
y_{1},y_{2}\right) $ where $\psi $ is a Lipschitz function, and $\left(
y_{1},y_{2},y_{3}\right) $ are rectangular coordinates in $R%
{{}^3}%
$ in a basis that may be different from the canonical basis $\left(
e_{1},e_{1},e_{3}\right) $.
\end{definition}

According to \cite{Tem1} one can draw the conclusion: It is useful for the
sequel of this section to note that a set $\Omega $ satisfying (1.4) is
"locally star-shaped". This means that each point $x_{j}\in \partial \Omega $%
, has an open neighbourhood $U_{j}$ such that $U_{j}^{\prime }=\Omega \cap
U_{j}$ is star-shaped with respect to one of its points. According to $%
\Omega $\textit{\ }is a locally Lipschitz we may, moreover, suppose that the
boundary $U_{j}^{\prime }$, $j\in J$ is Lipschitz, or $\partial \Omega \in
Lip_{loc}$. If $\partial \Omega $ is bounded, it can be covered by a finite
family of such sets $U_{j}^{\prime }$, $j\in J$. Consequently\textit{\ }for
every cross-section $\Omega _{L}\equiv \Omega \cap L\neq \varnothing $\ of $%
\Omega
$\ with arbitrary hyperplain $L$\ exists, at least, one coordinate subspace (%
$\left( x_{j},x_{k}\right) $)\ which possesses a domain $P_{x_{i}}\Omega
_{L} $\ (or union of domains) whit the Lipschitz class boundary since $%
\partial \Omega _{L}\equiv \partial \Omega \cap L\neq \varnothing $ and
isomorphically defining of $\Omega _{L}$\ with the affine representation,%
\textit{\ }in addition $\partial \Omega _{L}\Longleftrightarrow \partial
P_{x_{i}}\Omega _{L}$.

Thus, with use of the representation $P_{x_{i}}L$ of the hyperplane $L$ we
get that $\Omega _{L}$ can be written in the form $P_{x_{i}}\Omega _{L}$,
therefore an integral on $\Omega _{L}$ also will defined by the respective
representation, i. e. as the integral on $P_{x_{i}}\Omega _{L}$.

It should be noted that $\Omega _{L}$ can consist of many parts then $%
P_{x_{i}}\Omega _{L}$ will be such as $\Omega _{L}$. Consequently in this
case $\Omega _{L}$ will be as the union of domains and the following
relation 
\begin{equation*}
\Omega _{L}=\underset{r=1}{\overset{m}{\cup }}\Omega _{L}^{r}\
\Longleftarrow \Longrightarrow \ P_{x_{i}}\Omega _{L}=\underset{r=1}{\overset%
{m}{\cup }}P_{x_{i}}\Omega _{L}^{r},\quad \infty >m\geq 1,
\end{equation*}%
will holds by virtue of the definition \ref{D_3.1}. Therefore, each of $%
P_{x_{i}}\Omega _{L}^{r}$ will be the domain and one can investigate these
separately, because $\Omega _{L}^{r}\subset \Omega $ and $\partial \Omega
_{L}^{r}\subset \partial \Omega $ takes place.

So, we will define subdomains of $Q^{T}\equiv \left( 0,T\right) \times
\Omega $ as follows $Q_{L}^{T}\equiv \left( 0,T\right) \times \left( \Omega
\cap L\right) $, where $L$ is arbitrary fixed hyperplane of the dimension
two and $\Omega \cap L\neq \varnothing $. Therefore, we will study the
problem on the subdomain defined by use of the cross-section of $\Omega $ by
arbitrary fixed hyperplane dimension two $L$, i.e. by the $co\dim
_{R^{3}}L=1 $ ($\Omega \cap L$, namely on $Q_{L}^{T}\equiv \left( 0,T\right)
\times \left( \Omega \cap L\right) $).

Consequently, we will investigate uniqueness of the problem (1.1%
${{}^1}$%
) - (\ref{3}) on the "cross-section" $Q^{T}$ defined by the cross-section of 
$\Omega $, where $\Omega \subset R^{3}$. This cross-section we understand in
the following sense: Let $L$ be a hyperplane in $R^{3}$, i.e. with $co\dim
_{R^{3}}L=1$, that is equivalent to $R^{2}$. We denote by $\Omega _{L}$ the
cross-section of the form $\Omega _{L}\equiv \Omega \cap L\neq \varnothing $%
, $mes_{R^{2}}\left( \Omega _{L}\right) >0$, in the particular case $L\equiv
\left( x_{1},x_{2},0\right) $. In other words, if $L$ is the hyperplane in $%
R^{3}$ then we can determine it as $L\equiv \left\{ x\in R^{3}\left\vert \
a_{1}x_{1}+a_{2}x_{2}+a_{3}x_{3}=b\right. \right\} $, where coefficients $%
a_{i},b\in R^{1}$ ($i=1,2,3)$ are the arbitrary fixed constants. Whence
follows, that $a_{3}x_{3}=b-a_{1}x_{1}-a_{2}x_{2}$ or $x_{3}=\frac{1}{a_{3}}%
\left( b-a_{1}x_{1}-a_{2}x_{2}\right) $ if we assume $a_{i}\neq 0$ ($i=1,2,3$%
), or if we take substitutions: $\frac{b}{a_{3}}\Longrightarrow b,\frac{a_{1}%
}{a_{3}}\Longrightarrow a_{1}$ and $\frac{a_{2}}{a_{3}}\Longrightarrow a_{2}$
we derive $x_{3}\equiv \psi _{3}\left( x_{1},x_{2}\right)
=b-a_{1}x_{1}-a_{2}x_{2}$ in the new coefficients.

Thus, we have 
\begin{equation}
D_{3}\equiv \frac{\partial x_{1}}{\partial x_{3}}D_{1}+\frac{\partial x_{2}}{\partial x_{3}}D_{2}=-a_{1}^{-1}D_{1}-a_{2}^{-1}D_{2}\quad \&
\label{3.1}
\end{equation}%
\begin{equation}
D_{3}^{2}=a_{1}^{-2}D_{1}^{2}+a_{2}^{-2}D_{2}^{2}+2a_{1}^{-1}a_{2}^{-1}D_{1}D_{2},\quad D_{i}=\frac{\partial }{\partial x_{i}},i=1,2,3.
\label{3.2}
\end{equation}

For the application of our approach we need to assume that functions $u_{0}$
and $f$\ posseses some smoothness.

Now we take account the following conditions of Theorem \ref{Th_1} $u_{0}\in 
\mathcal{H}^{1/2}\left( \Omega \right) $, $f\in L^{2}\left( 0,T;\mathcal{H}%
^{1/2}\left( \Omega \right) \right) $ hold, consequently their restrictions
on $\left[ 0,T\right) \times \Omega _{L}$ are defined.

Let $L$ be arbitrary hyperplane such that $\Omega _{L}\neq \varnothing $ and 
$u\in \mathcal{V}\left( Q^{T}\right) $ be a solution of the problem (1.1%
${{}^1}$%
) - (\ref{3}). It is need to note the restriction of problem (1.1%
${{}^1}$%
) - (\ref{3}) to $\left[ 0,T\right) \times \Omega _{L}$ mean the restriction
on $\left[ 0,T\right) \times \Omega _{L}$ of a solution $u$ that is defined
on $\left[ 0,T\right) \times \Omega $. As function $u$ belong to $\mathcal{V}%
\left( Q^{T}\right) $ therefore the restriction $u$ on $\left[ 0,T\right)
\times \Omega _{L}$ is well defined. Then by making the restriction we
obtain the following problem on $\left[ 0,T\right) \times \Omega _{L}$, by
virtue of the above conditions of the main theorem, here $T>0$ some number, 
\begin{equation*}
\frac{\partial u}{\partial t}-\nu \Delta u+\underset{j=1}{\overset{3}{\sum }}%
u_{j}D_{j}u=\frac{\partial u_{L}}{\partial t}-\nu \left(
D_{1}^{2}+D_{2}^{2}+D_{3}^{2}\right) u_{L}+
\end{equation*}%
\begin{equation*}
u_{L1}D_{1}u_{L}+u_{L2}D_{2}u_{L}+u_{L3}D_{3}u_{L}=\frac{\partial u_{L}}{%
\partial t}-\nu \left[ D_{1}^{2}+D_{2}^{2}+a_{1}^{-2}D_{1}^{2}\right. +
\end{equation*}%
\begin{equation*}
\left. a_{2}^{-2}D_{2}^{2}+2a_{1}^{-1}a_{2}^{-1}D_{1}D_{2}\right]
u_{L}+u_{L1}D_{1}u_{L}+u_{L2}D_{2}u_{L}-u_{L3}a_{1}^{-1}D_{1}u_{L}-
\end{equation*}%
\begin{equation*}
u_{L3}a_{2}^{-1}D_{2}u_{L}=\frac{\partial u_{L}}{\partial t}-\nu \left[
\left( 1+a_{1}^{-2}\right) D_{1}^{2}+\left( 1+a_{2}^{-2}\right) D_{2}^{2}%
\right] u_{L}-
\end{equation*}%
\begin{equation}
2\nu a_{1}^{-1}a_{2}^{-1}D_{1}D_{2}u_{L}+\left(
u_{L1}-a_{1}^{-1}u_{L3}\right) D_{1}u_{L}+\left(
u_{L2}-a_{2}^{-1}u_{L3}\right) D_{2}u_{L}=f_{L}  \label{3.3}
\end{equation}%
on $\left( 0,T\right) \times \Omega _{L}$, by virtue of (\ref{3.1}) and (\ref%
{3.2}). We get 
\begin{equation}
\func{div}u_{L}=D_{1}\left( u_{L}-a_{1}^{-1}u_{L3}\right) +D_{2}\left( u_{L}-a_{2}^{-1}u_{L3}\right) =0,\quad x\in \Omega _{L},\ t>0
\label{3.4}
\end{equation}%
\begin{equation}
u_{L}\left( 0,x\right) =u_{L0}\left( x\right) ,\quad \left( t,x\right) \in \left[ 0,T\right] \times \Omega _{L};\quad u_{L}\left\vert \ _{\left( 0,T\right) \times \partial \Omega _{L}}\right. =0.
\label{3.5}
\end{equation}%
by using of same way.

Consequently, we restricted the problem (1.1%
${{}^1}$%
) - (\ref{3}) to problem (\ref{3.3}) - (\ref{3.5}) the study of which give
we possibility to define properties of solutions $u$ of problem (1.1%
${{}^1}$%
) - (\ref{3}) on each cross-section $\left[ 0,T\right) \times \Omega
_{L}\equiv Q_{L}^{T}$.

In the beginning it is necessary to investigate the existence of the
solution of problem (\ref{3.3}) - (\ref{3.5}) and determine the space where
the existing solutions are contained. Consequently, for ending the proof of
the uniqueness theorem for main problem it is enough to prove the existence
theorem and the uniqueness theorem for the derived problem (\ref{3.3}) - (%
\ref{3.5}), in this case. So now we will investigate of problem (\ref{3.3})
- (\ref{3.5}).

We would like to note: Let $L\subset R^{3}$ is the hyperplane for which $%
\Omega \cap L\neq \varnothing $ then there is, at least, one 2-dimensional
subspace in the given coordinat system that one can determine as $\left(
x_{i},x_{j}\right) $, consequently, $P_{x_{k}}L=R^{2}$, ($i,j,k=1,2,3$) i.e. 
\begin{equation*}
L\equiv \left\{ x\in R^{3}\left\vert \ x=\left( x_{i},x_{j},\psi _{L}\left(
x_{i},x_{j}\right) \right) ,\right. \left( x_{i},x_{j}\right) \in
R^{2}\right\}
\end{equation*}%
and 
\begin{equation*}
\Omega \cap L\equiv \left\{ x\in \Omega \left\vert \ x=\left(
x_{i},x_{j},\psi _{L}\left( x_{i},x_{j}\right) \right) ,\right. \left(
x_{i},x_{j}\right) \in P_{x_{k}}\left( \Omega \cap L\right) \right\}
\end{equation*}%
hold, where $\psi _{L}$ is the affine function such as the mentioned above
function and it is the bijection.

Thereby, in this case by applying of the mentioned restriction to functions 
\begin{equation*}
u(t,x_{1},x_{2},x_{3}),\ f(t,x_{1},x_{2},x_{3}),\ u_{0}(x_{1},x_{2},x_{3})
\end{equation*}%
we obtain the following representations 
\begin{equation*}
u(t,x_{i},x_{j},\psi _{L}(x_{i},x_{j}))\equiv v(t,x_{i},x_{j})\text{, }%
f(t,x_{i},x_{j},\psi _{L}(x_{i},x_{j})\equiv \phi (x_{i},x_{j})
\end{equation*}%
and%
\begin{equation*}
u_{0}(x_{i},x_{j},\psi _{L}(x_{i},x_{j}))\equiv v_{0}(x_{i},x_{j})\text{ \ \
on }(0,T)\times P_{x_{k}}\Omega _{L},
\end{equation*}%
respectively.

So, each of the functions obtained by the previous transformation depends
only on the indepandent variables: $t$, $x_{i}$ and $x_{j}$.

\subsection{\label{SS_4.1}\textbf{On Dirichlet to Neumann map}}

As is known (\cite{Nac}, \cite{BehEl}, \cite{DePrZac}, \cite{H-DR}, \cite%
{BelCho} etc.) the Dirichlet to Neumann map is single-value maping if the
homogeneous Dirichlet problem for elliptic equation has only trivial
solution, i. e. zero not is eigenvalue of this problem. Consequently, we
need to show that the homogeneous Dirichlet problem for elliptic equation
appropriate to considered problem satisfies of this property. So, we will
show that the obtained here problem satisfies of the corresponding condition
of results of such type from mentioned articles.

\begin{proposition}
\label{Pr_4.1}The homogeneous Dirichlet. problem for elliptic part of
problem (3.3) - (3.5) has only trivial solution.
\end{proposition}

\begin{proof}
If consider the elliptic part of problem (3.3) - (3.5) then we get\ the
problem 
\begin{equation*}
-\Delta u_{L}+Bu_{L}\equiv -\nu \left[ \left( 1+a_{1}^{-2}\right)
D_{1}^{2}+\left( 1+a_{2}^{-2}\right)
D_{2}^{2}+2a_{1}^{-1}a_{2}^{-1}D_{1}D_{2}\right] u_{Li}+
\end{equation*}%
\begin{equation*}
\left( u_{L1}-a_{1}^{-1}u_{L3}\right) D_{1}u_{L}+\left(
u_{L2}-a_{2}^{-1}u_{L3}\right) D_{2}u_{L}=0,\ x\in \Omega _{L},\quad
u_{L}\left\vert _{\ \partial \Omega _{L}}\right. =0,
\end{equation*}%
where $\Omega _{L}=\Omega \cap L$.

We assume this problem has a nontrivial solution and we will show that it is
unpossible. Let $u_{L}\in V\left( \Omega _{L}\right) $ be nontrivial
solution of this problem then we get the following equation 
\begin{equation*}
0=\left\langle -\Delta u_{L}+Bu_{L},u_{L}\right\rangle _{P_{x_{3}}\Omega
_{L}}
\end{equation*}%
hence 
\begin{equation*}
=-\underset{i=1}{\overset{3}{\nu \sum }}\left\langle \left[ \left(
D_{1}^{2}+D_{2}^{2}\right) +\left( a_{1}^{-1}D_{1}+a_{2}^{-1}D_{2}\right)
^{2}\right] u_{Li},u_{Li}\right\rangle _{P_{x_{3}}\Omega _{L}}+
\end{equation*}%
\begin{equation*}
\underset{i=1}{\overset{3}{\sum }}\underset{P_{x_{3}}\Omega _{L}}{\int }%
\left[ u_{L1}D_{1}u_{Li}u_{Li}+u_{L2}D_{2}u_{Li}u_{Li}+\right.
\end{equation*}%
\begin{equation*}
\left. u_{L3}\left( -a_{1}^{-1}D_{1}-a_{2}^{-1}D_{2}\right) u_{Li}u_{Li} 
\right] dx_{1}dx_{2}=
\end{equation*}%
\begin{equation*}
\underset{i=1}{\overset{3}{\nu \sum }}\underset{P_{x_{3}}\Omega _{L}}{\int }%
\left\{ \left( D_{1}u_{Li}\right) ^{2}+\left( D_{2}u_{Li}\right) ^{2}+\left[
\left( a_{1}^{-1}D_{1}+a_{2}^{-1}D_{2}\right) u_{Li}\right] ^{2}\right\}
dx_{1}dx_{2}+
\end{equation*}%
\begin{equation*}
\frac{1}{2}\underset{i=1}{\overset{3}{\sum }}\underset{P_{x_{3}}\Omega _{L}}{%
\int }\left[ u_{L1}D_{1}\left( u_{Li}\right) ^{2}+u_{L2}D_{2}\left(
u_{Li}\right) ^{2}+\right.
\end{equation*}%
\begin{equation*}
\left. u_{L3}\left( -a_{1}^{-1}D_{1}-a_{2}^{-1}D_{2}\right) \left(
u_{Li}\right) ^{2}\right] dx_{1}dx_{2}\geq
\end{equation*}%
\begin{equation*}
\underset{i=1}{\overset{3}{\nu \sum }}\underset{P_{x_{3}}\Omega _{L}}{\int }%
\left[ \left\vert D_{1}u_{Li}\right\vert ^{2}+\left\vert
D_{2}u_{Li}\right\vert ^{2}\right] dx_{1}dx_{2}+
\end{equation*}%
\begin{equation*}
-\frac{1}{2}\underset{i=1}{\overset{3}{\sum }}\underset{P_{x_{3}}\Omega _{L}}%
{\int }\left[ D_{1}u_{L1}+D_{2}u_{L2}+\left(
-a_{1}^{-1}D_{1}-a_{2}^{-1}D_{2}\right) u_{L3}\right] \left\vert
u_{Li}\right\vert ^{2}dx_{1}dx_{2}=
\end{equation*}%
by (\ref{3.4}) 
\begin{equation*}
\underset{i=1}{\overset{3}{\nu \sum }}\underset{P_{x_{3}}\Omega _{L}}{\int }%
\left[ \left\vert D_{1}u_{Li}\right\vert ^{2}+\left\vert
D_{2}u_{Li}\right\vert ^{2}\right] dx_{1}dx_{2}-\frac{1}{2}\underset{i=1}{%
\overset{3}{\sum }}\underset{P_{x_{3}}\Omega _{L}}{\int }\left\vert
u_{Li}\right\vert ^{2}\func{div}u_{L}dx_{1}dx_{2}=
\end{equation*}%
\begin{equation*}
\underset{i=1}{\overset{3}{\nu \sum }}\underset{P_{x_{3}}\Omega _{L}}{\int }%
\left[ \left\vert D_{1}u_{Li}\right\vert ^{2}+\left\vert
D_{2}u_{Li}\right\vert ^{2}\right] dx_{1}dx_{2}>0.
\end{equation*}%
Thus the obtained contradiction shows that function $u_{L}$ need be zero,
i.e. $u_{L}=0$ holds.

Consequently, the Dirichlet to Neumann map is single-value operator.
\end{proof}

It is well-known that operator $-\Delta :H_{0}^{1}\left( \Omega _{L}\right)
\longrightarrow $ $H^{-1}\left( \Omega _{L}\right) $ generates of the $C_{0}$
semigroup on $H\left( \Omega _{L}\right) $ and since the inclusion $%
H_{0}^{1}\left( \Omega _{L}\right) \subset H^{-1}\left( \Omega _{L}\right) $
is compact therefore $\left( -\Delta \right) ^{-1}$ is the compact operator
in $H^{-1}\left( \Omega _{L}\right) $. Moreover, $-\Delta :H^{1/2}\left(
\partial \Omega _{L}\right) \longrightarrow H^{-1/2}\left( \partial \Omega
_{L}\right) $ and the operator $B:$ $H^{1/2}\left( \partial \Omega
_{L}\right) \longrightarrow H^{-1/2}\left( \partial \Omega _{L}\right) $
also possess appropriate properties of such types.

\section{\label{Sec_5}Existence of Solution of Problem (3.3) - (3.5)}

So, assume conditons of Theorem \ref{Th_1} fulfilled, i. e. 
\begin{equation*}
u_{0}\in \mathcal{H}^{1/2}\left( \Omega \right) ,\quad f\in L^{2}\left( 0,T;%
\mathcal{H}^{1/2}\left( \Omega \right) \right) ,
\end{equation*}%
then restrictions of these functions on $\Omega _{L}$, $Q_{L}^{T}$,
respectively, are correctly defined and belong in $H\left( \Omega
_{L}\right) $, $L^{2}\left( 0,T;V^{\ast }\left( \Omega _{L}\right) \right) $%
, respectively. Consequently, it is enough to study the restricted problem
under conditions $u_{0L}\in H\left( \Omega _{L}\right) $ and $f_{L}\in
L^{2}\left( 0,T;V^{\ast }\left( \Omega _{L}\right) \right) $, as independent
problem.

To carry out the known argument started by Leray (\cite{Ler1}, see, also 
\cite{Lio1}, \cite{Tem1}) we can determine the following spaces 
\begin{equation*}
V\left( \Omega _{L}\right) =\left\{ v\left\vert \ v\in \right. \left(
W_{0}^{1,2}\left( \Omega _{L}\right) \right) ^{3}\equiv \left(
H_{0}^{1}\left( \Omega _{L}\right) \right) ^{3},\quad \func{div}v=0\right\} ,
\end{equation*}%
where $\func{div}$ is regarded in the sense (\ref{3.4}) and 
\begin{equation*}
V\left( Q_{L}^{T}\right) \equiv L^{2}\left( 0,T;V\left( \Omega _{L}\right)
\right) \cap L^{\infty }\left( 0,T;\left( H\left( \Omega _{L}\right) \right)
^{3}\right) .
\end{equation*}

More exactly we will adduce definitions of these spaces such way as in
Definition \ref{D_2.1} , i. e. $V\left( \Omega _{L}\right) $ is the closure
in $\left( H_{0}^{1}\left( \Omega _{L}\right) \right) ^{3}$\ of 
\begin{equation*}
\left\{ \varphi \left\vert \ \varphi \in \left( C_{0}^{\infty }\left( \Omega
_{L}\right) \right) ^{3},\right. \func{div}\varphi =0\right\}
\end{equation*}%
\ the dual $V\left( \Omega _{L}\right) $ is determined as $V^{\ast }\left(
\Omega _{L}\right) $ and $\left( H\left( \Omega _{L}\right) \right) ^{3}$ is
the closure in $\left( L^{2}\left( \Omega _{L}\right) \right) ^{3}$ of 
\begin{equation*}
\left\{ \varphi \left\vert \ \varphi \in \left( C_{0}^{\infty }\left( \Omega
_{L}\right) \right) ^{3},\right. \func{div}\varphi =0\right\} .
\end{equation*}

Here $\Omega \subset R^{3}$ is bounded domain of $Lip_{loc}$ and $\Omega
_{L}\subset R^{2}$ is subdomain defined in Section \ref{Sec_4} therefore, $%
\Omega _{L}$ is Lipschitz, $Q_{L}^{T}\equiv \left( 0,T\right) \times \Omega
_{L}$.

Consequently, a solution of this problem will be understood as follows: Let $%
f_{L}\in L^{2}\left( 0,T;V^{\ast }\left( \Omega _{L}\right) \right) $ and $%
u_{0L}\in \left( H\left( \Omega _{L}\right) \right) ^{3}$.

So, we can call the solution of this problem: a function $u_{L}\in \mathcal{V%
}\left( Q_{L}^{T}\right) $ is called a solution of the problem (\ref{3.3}) -
(\ref{3.5}) if $u_{L}(t,x^{\prime })$ satisfy the equation and initial
condition 
\begin{equation}
\frac{d}{dt}\left\langle u_{L},v\right\rangle _{\Omega _{L}}-\left\langle
\nu \Delta u_{L},v\right\rangle _{\Omega _{L}}+\left\langle \underset{j=1}{%
\overset{3}{\sum }}u_{Lj}D_{j}u_{L},v\right\rangle _{\Omega
_{L}}=\left\langle f_{L},v\right\rangle _{\Omega _{L}},  \label{3.6}
\end{equation}%
\begin{equation*}
\left\langle u_{L}\left( t\right) ,v\right\rangle \left\vert _{t=0}\right.
=\left\langle u_{0L},v\right\rangle ,
\end{equation*}%
for any $v\in V\left( \Omega _{L}\right) $ a. e. on $\left( 0,T\right) $ in
the sense of $H$, here $\left\langle \circ ,\circ \right\rangle _{\Omega
_{L}}$ is the dual form for the pair of spaces $\left( V\left( \Omega
_{L}\right) ,V^{\ast }\left( \Omega _{L}\right) \right) $ and $\Omega _{L}$
is Lipschitz, where $x^{\prime }\equiv \left( x_{1},x_{2}\right) $ and 
\begin{equation*}
\mathcal{V}\left( Q_{L}^{T}\right) \equiv \left\{ w\left\vert \ w\in V\left(
Q_{L}^{T}\right) ,\ w^{\prime }\in L^{2}\left( 0,T;V^{\ast }\left( \Omega
_{L}\right) \right) \right. \right\} .
\end{equation*}

We will lead of the proof of this problem in five-steps as indepandent
problem.

\subsection{\label{SS_5.1}\textbf{A priori estamations}}

For this we assume in (\ref{3.6}) $u_{L}$ instead of $v$ then we get 
\begin{equation}
\frac{d}{dt}\left\langle u_{L},u_{L}\right\rangle _{\Omega
_{L}}-\left\langle \nu \Delta u_{L},u_{L}\right\rangle _{\Omega
_{L}}+\left\langle \underset{j=1}{\overset{3}{\sum }}u_{Lj}D_{j}u_{L},u_{L}%
\right\rangle _{\Omega _{L}}=\left\langle f_{L},u_{L}\right\rangle _{\Omega
_{L}}.  \label{3.6'}
\end{equation}%
Thence, by making the known calculations and taking into account of the
condition on $\Omega _{L}$ and of calculations (\ref{3.1}) and next (\ref%
{3.4}) that carried out in previous Section, we obtain

\begin{equation*}
\frac{1}{2}\frac{d}{dt}\left\Vert u_{L}\right\Vert _{\left( H\left( \Omega
_{L}\right) \right) ^{3}}^{2}\left( t\right) +\nu \left( 1+a_{1}^{-2}\right)
\left\Vert D_{1}u_{L}\right\Vert _{\left( H\left( \Omega _{L}\right) \right)
^{3}}^{2}\left( t\right) +
\end{equation*}%
\begin{equation}
\nu \left( 1+a_{2}^{-2}\right) \left\Vert D_{2}u_{L}\right\Vert _{\left(
H\left( \Omega _{L}\right) \right) ^{3}}^{2}\left( t\right) +2\nu
a_{1}^{-1}a_{2}^{-1}\left\langle D_{1}u_{L},D_{2}u_{L}\right\rangle _{\Omega
_{L}}\left( t\right) =\left\langle f_{L},u_{L}\right\rangle _{\Omega _{L}},
\label{3.7}
\end{equation}%
where $\left\langle g,h\right\rangle _{\Omega _{L}}=\underset{i=1}{\overset{3%
}{\sum }}\underset{P_{x_{3}}\Omega _{L}}{\int }g_{i}h_{i}dx_{1}dx_{2}$ for
any $g,h\in \left( H\left( \Omega _{L}\right) \right) ^{3}$, or $g\in \left(
H^{1}\left( \Omega _{L}\right) \right) ^{3}$ and $h\in \left( H^{-1}\left(
\Omega _{L}\right) \right) ^{3}$, respectively. We will show the correctness
of (\ref{3.7}), and to this end we shall prove the correctness of each term
of this sum, separately.

So, using (\ref{3.3}) we get 
\begin{equation*}
-\nu \left\langle \Delta u_{L}\left( t\right) ,u_{L}\left( t\right)
\right\rangle _{\Omega _{L}}=
\end{equation*}%
\begin{equation*}
-\underset{i=1}{\overset{3}{\nu \sum }}\left\langle \left[ \left(
1+a_{1}^{-2}\right) D_{1}^{2}+\left( 1+a_{2}^{-2}\right)
D_{2}^{2}+2a_{1}^{-1}a_{2}^{-1}D_{1}D_{2}\right] u_{Li},u_{Li}\right\rangle
_{P_{x_{3}}\Omega _{L}}=
\end{equation*}%
\begin{equation*}
\underset{i=1}{\overset{3}{\nu \sum }}\underset{P_{x_{3}}\Omega _{L}}{\int }%
\left[ \left( 1+a_{1}^{-2}\right) \left( D_{1}u_{Li}\right) ^{2}+\left(
1+a_{2}^{-2}\right) \left( D_{2}u_{Li}\right) ^{2}+\right.
\end{equation*}%
\begin{equation*}
\left. 2a_{1}^{-1}a_{2}^{-1}D_{1}u_{Li}D_{2}u_{Li}\right] dx_{1}dx_{2}\geq
\end{equation*}%
thus is obtained the sum reducible in (\ref{3.7}); if we estimate of the
last adding in the above mentioned sum then we get 
\begin{equation}
\nu \left[ \left\Vert D_{1}u_{L}\right\Vert _{\left( H\left( \Omega
_{L}\right) \right) ^{3}}^{2}\left( t\right) +\left\Vert
D_{2}u_{L}\right\Vert _{\left( H\left( \Omega _{L}\right) \right)
^{3}}^{2}\left( t\right) \right] .  \label{3.8}
\end{equation}

Now consider the trilinear form from (\ref{3.6'}) 
\begin{equation*}
\left\langle \underset{j=1}{\overset{3}{\sum }}u_{Lj}D_{j}u_{L},u_{L}\right%
\rangle _{\Omega _{L}}=
\end{equation*}%
due to (\ref{3.3}) we get 
\begin{equation*}
\underset{i=1}{\overset{3}{\sum }}\underset{P_{x_{3}}\Omega _{L}}{\int }%
\left[ u_{L1}D_{1}u_{Li}u_{Li}+u_{L2}D_{2}u_{Li}u_{Li}+\right.
\end{equation*}%
\begin{equation*}
\left. u_{L3}\left( -a_{1}^{-1}D_{1}-a_{2}^{-1}D_{2}\right) u_{Li}u_{Li} 
\right] dx_{1}dx_{2}=
\end{equation*}%
\begin{equation*}
\frac{1}{2}\underset{i=1}{\overset{3}{\sum }}\underset{P_{x_{3}}\Omega _{L}}{%
\int }\left[ u_{L1}D_{1}\left( u_{Li}\right) ^{2}+u_{L2}D_{2}\left(
u_{Li}\right) ^{2}+\right.
\end{equation*}%
\begin{equation*}
\left. u_{L3}\left( -a_{1}^{-1}D_{1}-a_{2}^{-1}D_{2}\right) \left(
u_{Li}\right) ^{2}\right] dx_{1}dx_{2}=
\end{equation*}%
\begin{equation*}
-\frac{1}{2}\underset{i=1}{\overset{3}{\sum }}\underset{P_{x_{3}}\Omega _{L}}%
{\int }\left[ D_{1}u_{L1}+D_{2}u_{L2}+\left(
-a_{1}^{-1}D_{1}-a_{2}^{-1}D_{2}\right) u_{L3}\right] \left( u_{Li}\right)
^{2}dx_{1}dx_{2}=
\end{equation*}%
hence by (\ref{3.4}) 
\begin{equation}
-\frac{1}{2}\underset{i=1}{\overset{3}{\sum }}\underset{P_{x_{3}}\Omega _{L}}%
{\int }\left( u_{Li}\right) ^{2}\func{div}u_{L}dx_{1}dx_{2}=0.  \label{3.9}
\end{equation}

Consequently, the correctness of equation (\ref{3.7}) follows from (\ref{3.8}%
)-(\ref{3.9}), that give we the following inequation 
\begin{equation*}
\frac{1}{2}\frac{d}{dt}\left\Vert u_{L}\right\Vert _{\left( H\left( \Omega
_{L}\right) \right) ^{3}}^{2}\left( t\right) +
\end{equation*}%
\begin{equation}
\nu \underset{i=1}{\overset{3}{\sum }}\underset{P_{x_{3}}\Omega _{L}}{\int }%
\left[ \left( D_{1}u_{Li}\right) ^{2}+\left( D_{2}u_{Li}\right) ^{2}\right]
dx_{1}dx_{2}\leq \underset{P_{x_{3}}\Omega _{L}}{\int }\left\vert \left(
f_{L}\cdot u_{L}\right) \right\vert dx_{1}dx_{2}  \label{3.10}
\end{equation}%
Namely, from here we obtain the following a priori estimations 
\begin{equation}
\left\Vert u_{L}\right\Vert _{\left( H\left( \Omega _{L}\right) \right)
^{3}}\left( t\right) \leq C\left( f_{L},u_{L0},mes\Omega \right) ,
\label{3.11}
\end{equation}%
\begin{equation}
\left\Vert D_{1}u_{L}\right\Vert _{\left( H\left( \Omega _{L}\right) \right)
^{3}}+\left\Vert D_{2}u_{L}\right\Vert _{\left( H\left( \Omega _{L}\right)
\right) ^{3}}\leq C\left( f_{L},u_{L0},mes\Omega \right) ,  \label{3.12}
\end{equation}%
where $C\left( f_{L},u_{L0},mes\Omega \right) >0$ is the constant that is
independent of $u_{L}$. Consequently, any possible solution of this problem
belong to a bounded subset of the space $V\left( Q_{L}^{T}\right) $.

So, if we will obtain the estimation for $u_{L}^{\prime }$ as well then we
will have of the necessary a priori estimations, which are sufficient for
the proof of the existence theorem.\footnote{%
It should be noted that if the represantation of $%
\Omega
_{L}$ by coordinate system $(x_{1},x_{2})$ not is best for the definition of
the appropriate integral, then we will select other coordinate system:
either $(x_{1},x_{3})$ or $(x_{2},x_{3})$ instead of $(x_{1},x_{2})$ that is
best for our goal, that must exist by virtue of the definition of $%
\Omega
$.}

\subsection{\label{SS_5.2}Boundedness of the trilinear form}

Now we must study the trilinear form of (\ref{3.6}) that one can also call
as $b_{L}\left( u_{L},u_{L},v\right) $.

\begin{proposition}
\label{P_3.1}Let $u_{L}\in V\left( Q_{L}^{T}\right) $, $v\in V\left( \Omega
_{L}\right) $ and $B$ is the operator defined by 
\begin{equation*}
\left\langle B\left( u_{L}\right) ,v\right\rangle _{\Omega _{L}}=b_{L}\left(
u_{L},u_{L},v\right) =\left\langle \underset{j=1}{\overset{3}{\sum }}%
u_{Lj}D_{j}u_{L},v\right\rangle _{\Omega _{L}}
\end{equation*}%
then $B\left( u_{L}\right) $ belong to $L^{2}\left( 0,T;V^{\ast }\left(
\Omega _{L}\right) \right) $.
\end{proposition}

\begin{proof}
At first we will show boundedness of the operator $B$ from $V\left( \Omega
_{L}\right) $ to $V^{\ast }\left( \Omega _{L}\right) $ for a. e. $t\in
\left( 0,T\right) $. We have 
\begin{equation*}
\left\langle B\left( u_{L}\right) ,v\right\rangle _{\Omega
_{L}}=\left\langle \underset{j=1}{\overset{3}{\sum }}u_{Lj}D_{j}u_{L},v%
\right\rangle _{\Omega _{L}}=
\end{equation*}%
due of (\ref{3.4}) and of the definition \ref{D_3.1} 
\begin{equation*}
\underset{i=1}{\overset{3}{\sum }}\underset{P_{x_{3}}\Omega _{L}}{\int }%
\left[ u_{L1}D_{1}u_{Li}v_{i}+u_{L2}D_{2}u_{Li}v_{i}+u_{L3}\left(
-a_{1}^{-1}D_{1}-a_{2}^{-1}D_{2}\right) u_{Li}v_{i}\right] dx_{1}dx_{2}=
\end{equation*}%
\begin{equation*}
-\underset{i=1}{\overset{3}{\sum }}\underset{P_{x_{3}}\Omega _{L}}{\int }%
\left[ u_{L1}u_{Li}D_{1}v_{i}+u_{L2}u_{Li}D_{2}v_{i}+u_{L3}u_{Li}\left(
-a_{1}^{-1}D_{1}-a_{2}^{-1}D_{2}\right) v_{i}\right] dx_{1}dx_{2}=
\end{equation*}%
\begin{equation}
-\underset{i=1}{\overset{3}{\sum }}\underset{P_{x_{3}}\Omega _{L}}{\int }%
u_{Li}\left[ \left( u_{L1}-a_{1}^{-1}u_{L3}\right) D_{1}v_{i}+\left(
u_{L2}-a_{2}^{-1}u_{L3}\right) D_{2}v_{i}\right] dx_{1}dx_{2}.  \label{3.13}
\end{equation}

Hence follows 
\begin{equation*}
\left\vert \left\langle B\left( u_{L}\right) ,v\right\rangle \right\vert
\leq \underset{i=1}{\overset{3}{\sum }}\underset{P_{x_{3}}\Omega _{L}}{\int }%
\left\vert u_{L}\right\vert ^{2}\left( \left\vert D_{1}v_{i}\right\vert
+\left\vert D_{2}v_{i}\right\vert \right) dx_{1}dx_{2}\leq
\end{equation*}%
\begin{equation}
c\left\Vert u_{L}\right\Vert _{L^{4}\left( \Omega _{L}\right)
}^{2}\left\Vert v\right\Vert _{V\left( \Omega _{L}\right) }\Longrightarrow
\left\Vert B\left( u_{L}\right) \right\Vert _{V^{\ast }}\leq c\left\Vert
u_{L}\right\Vert _{V}^{2}  \label{3.14}
\end{equation}%
due of $V\left( \Omega _{L}\right) \subset L^{4}\left( \Omega _{L}\right) $.
This also shows that operator $B:V\left( \Omega _{L}\right) \longrightarrow
V^{\ast }\left( \Omega _{L}\right) $ is continuous for a. e. $t>0$.

Finally, we obtain needed result using above mentioned inequality and the
well-known inequality (see, \cite{Lad1}, \cite{Lio1}, \cite{Tem1}), which is
correct for the space with two dimension 
\begin{equation*}
\overset{T}{\underset{0}{\int }}\left\Vert B\left( u_{L}\left( t\right)
\right) \right\Vert _{V^{\ast }}^{2}dt\leq c\overset{T}{\underset{0}{\int }}%
\left\Vert u_{L}\left( t\right) \right\Vert _{L^{4}}^{4}dt\leq c_{1}\overset{%
T}{\underset{0}{\int }}\left\Vert u_{L}\left( t\right) \right\Vert
_{L^{2}}^{2}\left\Vert u_{L}\right\Vert _{V}^{2}dt\leq
\end{equation*}%
\begin{equation*}
c_{1}\left\Vert u_{L}\right\Vert _{L^{\infty }\left( 0,T;H\right) }^{2}%
\overset{T}{\underset{0}{\int }}\left\Vert u_{L}\right\Vert
_{V}^{2}dt\Longrightarrow
\end{equation*}%
\begin{equation}
\left\Vert B\left( u_{L}\right) \right\Vert _{L^{2}\left( 0,T;V^{\ast
}\right) }\leq c_{1}\left\Vert u_{L}\right\Vert _{L^{\infty }\left(
0,T;H\right) }\left\Vert u_{L}\right\Vert _{L^{2}\left( 0,T;V\right) }.
\label{3.15}
\end{equation}

What was to be proved.
\end{proof}

\subsection{\label{SS_5.3}Boundedness of $u^{\prime }$}

Sketch of the proof of the inclusion: $u^{\prime }$ belong to bounded subset
of $L^{2}\left( 0,T;V^{\ast }\left( \Omega _{L}\right) \right) $. It is
possible to draw the following conclusion based on receiving of a priori
estimates, on proposition \ref{P_3.1} and on reflexivity of all used spaces:
If we were used of the Faedo-Galerkin's method for investigation we could
obtain estimations for the approximate solutions the same as \ref{3.11}, \ref%
{3.12} and \ref{3.15}. Since $V\left( \Omega _{L}\right) $ is a separable
there exists a sequence of linearly independent elements $\left\{
w_{i}\right\} _{i=1}^{\infty }\subset V\left( \Omega _{L}\right) $, which is
total in $V\left( \Omega _{L}\right) $. For each $m$ we define an
approximate solution $u_{Lm}$ of (\ref{3.3}) or (\ref{3.6}) as follows:

\begin{equation}
u_{Lm}=\overset{m}{\underset{i=1}{\sum }}u_{Lm}^{i}\left( t\right)
w_{i},\quad m=1,\ 2,....  \label{5.4}
\end{equation}%
where $u_{Lm}^{i}\left( t\right) $, $i=\overline{1,m}$ be unknown functions
that will be determined as solutions of following system of the differential
equations that is received according to equation (\ref{3.6}) 
\begin{equation*}
\left\langle \frac{d}{dt}u_{Lm},w_{j}\right\rangle _{\Omega
_{L}}=\left\langle \nu \Delta u_{Lm},w_{j}\right\rangle _{\Omega
_{L}}+b_{L}\left( u_{Lm},u_{Lm},w_{j}\right) +
\end{equation*}%
\begin{equation}
\left\langle f_{L},w_{j}\right\rangle _{\Omega _{L}},\quad t\in \left( 0,T 
\right] ,\quad j=\overline{1,m},\quad  \label{5.5}
\end{equation}%
\begin{equation*}
u_{Lm}\left( 0\right) =u_{0Lm}
\end{equation*}%
where $\left\{ u_{0Lm}\right\} _{m=1}^{\infty }\subset H\left( \Omega
_{L}\right) $ is some sequence such that $u_{0Lm}\longrightarrow u_{0L}$ in $%
H\left( \Omega _{L}\right) $ as $m\longrightarrow \infty $. (Since $V\left(
\Omega _{L}\right) $ is everywhere dense in $H\left( \Omega _{L}\right) $ we
can determine $u_{0Lm}$ by using the total system $\left\{ w_{i}\right\}
_{i=1}^{\infty }$). \ 

With use (\ref{5.4}) in (\ref{5.5}) we have 
\begin{equation*}
\overset{m}{\underset{j=1}{\sum }}\left\langle w_{j},w_{i}\right\rangle
_{\Omega _{L}}\frac{d}{dt}u_{Lm}^{j}\left( t\right) -\nu \overset{m}{%
\underset{j=1}{\sum }}\left\langle \Delta w_{j},w_{i}\right\rangle _{\Omega
_{L}}u_{Lm}^{j}\left( t\right) +
\end{equation*}%
\begin{equation*}
\overset{m}{\underset{j,k=1}{\sum }}b_{L}\left( w_{j},w_{k},w_{i}\right)
u_{Lm}^{j}\left( t\right) u_{Lm}^{k}\left( t\right) =\left\langle
f_{L}\left( t\right) ,w_{i}\right\rangle _{\Omega _{L}},\ i=\overline{1,m}
\end{equation*}%
and taking into account the nonsingularity of matrix $\left\langle
w_{i},w_{j}\right\rangle _{\Omega _{L}}$, $i,j=\overline{1,m}$ we get to
system of differential equations for $u_{Lm}^{i}\left( t\right) $, $%
i=1,...,m $ 
\begin{equation*}
\frac{du_{Lm}^{i}\left( t\right) }{dt}=\overset{m}{\underset{j=1}{\sum }}%
c_{i,j}\left\langle f_{L}\left( t\right) ,w_{j}\right\rangle _{\Omega
_{L}}-\nu \overset{m}{\underset{j=1}{\sum }}d_{i,j}u_{Lm}^{j}\left( t\right)
+
\end{equation*}%
\begin{equation}
\overset{m}{\underset{j,k=1}{\sum }}h_{ijk}u_{Lm}^{j}\left( t\right)
u_{Lm}^{k}\left( t\right) ,  \label{5.6}
\end{equation}%
\begin{equation*}
u_{Lm}^{i}\left( 0\right) =u_{0Lm}^{i},\quad i=1,...,m,\quad m=1,\ 2,...
\end{equation*}%
here $u_{0Lm}^{i}$ is $i^{th}$ component of $u_{0L}$ of the representation $%
u_{0L}=\overset{\infty }{\underset{k=1}{\sum }}u_{0Lm}^{k}w_{k}$.

The Cauchy problem for the nonlinear differential system (\ref{5.6}) has
solution, which is defined on the whole interval $(0,T]$ by virtue of
uniformity of estimations received in subsections \ref{SS_5.1} and \ref%
{SS_5.2}. Consequently, the approximate solution $u_{Lm}$ belong to a
bounded subset of $W^{1,2}\left( 0,T;V^{\ast }\left( \Omega _{L}\right)
\right) $ for every $m=1,\ 2,...$ since the right side of (\ref{5.6}) belong
to the bounded subset of $L^{2}\left( 0,T;V^{\ast }\left( \Omega _{L}\right)
\right) $ as were proved in subsections \ref{SS_5.1} and \ref{SS_5.2}, by
virtue of the lemma that is adduced below.

\begin{lemma}
(\cite{Tem1}) Let $X$ be a given Banach space with dual $X^{\ast }$ and let $%
u$ and $g$ be two functions belonging to $L^{1}\left( a,b;X\right) $. Then,
the following three conditions are equivalent

\textit{(i)} $u$ is a. e. equal to a primitive function of $g$, 
\begin{equation*}
u\left( t\right) =\xi +\underset{a}{\overset{t}{\int }}g\left( s\right)
ds,\quad \xi \in X,\quad \text{a.e. }t\in \left[ a,b\right]
\end{equation*}%
\textit{(ii)} For each test function $\varphi \in D\left( \left( a,b\right)
\right) $, 
\begin{equation*}
\underset{a}{\overset{b}{\int }}u\left( t\right) \varphi ^{\prime }\left(
t\right) dt=-\underset{a}{\overset{b}{\int }}g\left( t\right) \varphi \left(
t\right) dt,\quad \varphi ^{\prime }=\frac{d\varphi }{dt}
\end{equation*}%
(iii) For each $\eta \in X^{\ast }$, 
\begin{equation*}
\frac{d}{dt}\left\langle u,\eta \right\rangle =\left\langle g,\eta
\right\rangle
\end{equation*}%
in the scalar distribution sense, on $(a,b)$. If \textit{(i) - (iii)} are
satisfied $u$, in particular, is a. e. equal to a continuous function from $%
[a,b]$ into $X$.
\end{lemma}

It not is difficult to see that if passing to limit at $m\longrightarrow
\infty $ in equation (\ref{5.5}) (maybe by a subsequence $\left\{ u_{Lm_{%
\mathit{l}}}\right\} _{\mathit{l}=1}^{\infty }$ as is known such subsequence
exists) and to take $\forall v\in V\left( \Omega _{L}\right) $ instead of $%
w_{k}$\ then we get \ 
\begin{equation*}
\left\langle \frac{d}{dt}u_{L},v\right\rangle _{\Omega _{L}}=\left\langle
f_{L}+\nu \Delta u_{L}-\chi ,v\right\rangle _{\Omega _{L}},
\end{equation*}%
as $\left\{ w_{i}\right\} _{i=1}^{\infty }$ is total in $V\left( \Omega
_{L}\right) $, where $\chi $ belong to $L^{2}\left( 0,T;V^{\ast }\left(
\Omega _{L}\right) \right) $ and is determined 
\begin{equation*}
\underset{\mathit{l}\longrightarrow \infty }{\lim }\left\langle B\left(
u_{Lm_{\mathit{l}}}\right) ,v\right\rangle _{\Omega _{L}}=\left\langle \chi
,v\right\rangle _{\Omega _{L}}.
\end{equation*}%
So, according to above a priori estimations and Proposition \ref{P_3.1} we
obtain that the right side belong to $L^{2}\left( 0,T\right) $ then the left
side also belongs to $L^{2}\left( 0,T\right) $, i. e. 
\begin{equation*}
\frac{du_{L}}{dt}\in L^{2}\left( 0,T;V^{\ast }\left( \Omega _{L}\right)
\right) .
\end{equation*}%
Consequently, the function $u_{L}$ belong to a bounded subset of the space $%
\mathcal{V}\left( Q_{L}^{T}\right) $, where 
\begin{equation}
\mathcal{V}\left( Q_{L}^{T}\right) \equiv V\left( Q_{L}^{T}\right) \cap
W^{1,2}\left( 0,T;V^{\ast }\left( \Omega _{L}\right) \right) ,  \label{3.16}
\end{equation}%
by virtue of the above mentioned lemma and abstract form of Riesz-Fischer
theorem.

Thus for the proof that $u_{l}$ is the solution of equation (\ref{3.3}) or (%
\ref{3.6}) remains to prove that $\chi =B\left( u_{L}\right) $ or $%
\left\langle \chi ,v\right\rangle _{\Omega _{L}}=b_{L}\left(
u_{L},u_{L},v\right) $ for $\forall v\in V\left( \Omega _{L}\right) $. \ 

\subsection{\label{SS_5.4}Weakly compactness of operator $B$}

\begin{proposition}
\label{P_3.2}Operator $B:\mathcal{V}\left( Q_{L}^{T}\right) \longrightarrow
L^{2}\left( 0,T;V^{\ast }\left( \Omega _{L}\right) \right) $ is weakly
compact operator, i. e. any weakly convergent sequence $\left\{
u_{L}^{m}\right\} _{1}^{\infty }\subset \mathcal{V}\left( Q_{L}^{T}\right) $
possesses a subsequence $\left\{ u_{L}^{m_{k}}\right\} _{1}^{\infty }\subset
\left\{ u_{L}^{m}\right\} _{1}^{\infty }$ such that $\left\{ B\left(
u_{L}^{m_{k}}\right) \right\} _{1}^{\infty }$ weakly converge in $%
L^{2}\left( 0,T;V^{\ast }\left( \Omega _{L}\right) \right) $.
\end{proposition}

\begin{proof}
Let a sequence $\left\{ u_{L}^{m}\right\} _{1}^{\infty }\subset \mathcal{V}%
\left( Q_{L}^{T}\right) $ be weakly converge to $u_{L}^{0}$ in $\mathcal{V}%
\left( Q_{L}^{T}\right) $. Then there exists a subsequence $\left\{
u_{L}^{m_{k}}\right\} _{1}^{\infty }\subset \left\{ u_{L}^{m}\right\}
_{1}^{\infty }$ such that $u_{L}^{m_{k}}\longrightarrow u_{L}^{0}$ in $%
L^{2}\left( 0,T;H\right) $ according to known embedding theorems, i. e.
since the inclusion 
\begin{equation*}
L^{2}\left( 0,T;V\left( \Omega _{L}\right) \right) \cap W^{1,2}\left(
0,T;V^{\ast }\left( \Omega _{L}\right) \right) \subset L^{2}\left(
0,T;H\right)
\end{equation*}%
is compact (see, e. g. \cite{Lio1}, \cite{Tem1}). Indeed, for us it is
enough to show that the operator generated by $\underset{j=1}{\overset{3}{%
\sum }}u_{Lj}D_{j}u_{L}$ is weakly compact from $\mathcal{V}\left(
Q_{L}^{T}\right) $ to $L^{2}\left( 0,T;V^{\ast }\left( \Omega _{L}\right)
\right) $. From a priori estimations and Proposition \ref{P_3.1} follows
that operator $B:\mathcal{V}\left( Q_{L}^{T}\right) \longrightarrow
L^{2}\left( 0,T;V^{\ast }\left( \Omega _{L}\right) \right) $ is bounded,
consequently $B\left( \left\{ u_{L}^{m_{k}}\right\} _{1}^{\infty }\right) $
belongs to bounded subset of the space $L^{2}\left( 0,T;V^{\ast }\left(
\Omega _{L}\right) \right) $. This lead the weak convergence 
\begin{equation}
B\left( u_{L}^{m_{k}}\right) \rightharpoonup \chi \quad \text{ in }%
L^{2}\left( 0,T;V^{\ast }\left( \Omega _{L}\right) \right)  \label{3.17}
\end{equation}%
according by reflexivity of this space (at least, there exists such
subsequence that this occurs).

Introduce the vector space 
\begin{equation*}
\mathcal{C}^{1}\left( \overline{Q}_{L}\right) \equiv \left\{ v\left\vert \
v_{i}\in C^{1}\left( \left[ 0,T\right] ;C_{0}^{1}\left( \overline{\Omega _{L}%
}\right) \right) ,\right. i=1,2,3\right\}
\end{equation*}%
and consider the trilinear form 
\begin{equation*}
\underset{0}{\overset{T}{\int }}\left\langle B\left( u_{L}^{m}\right)
,v\right\rangle _{\Omega _{L}}dt=\underset{0}{\overset{T}{\int }}b\left(
u_{L}^{m},u_{L}^{m},v\right) dt=\underset{0}{\overset{T}{\int }}\left\langle 
\underset{j=1}{\overset{3}{\sum }}u_{Lj}^{m}D_{j}u_{L}^{m},v\right\rangle
_{\Omega _{L}}dt=
\end{equation*}%
where $v\in \mathcal{C}^{1}\left( \overline{Q}_{L}\right) $, then according
to (\ref{3.13}) we get 
\begin{equation*}
-\underset{i=1}{\overset{3}{\sum }}\underset{0}{\overset{T}{\int }}\underset{%
P_{x_{3}}\Omega _{L}}{\int }\left[ \left(
u_{Li}^{m}u_{L1}^{m}-a_{1}^{-1}u_{Li}^{m}u_{L3}^{m}\right) D_{1}v_{i}+\left(
u_{Li}^{m}u_{L2}^{m}-a_{2}^{-1}u_{Li}^{m}u_{L3}^{m}\right) D_{2}v_{i}\right]
dx_{1}dx_{2}dt.
\end{equation*}

If in this sum separately we take arbitrary therm then it is not difficult
to see that the following convergence is true 
\begin{equation*}
\underset{0}{\overset{T}{\int }}\underset{P_{x_{3}}\Omega _{L}}{\int }%
u_{Li}^{m}u_{L1}^{m}D_{1}v_{i}dx_{1}dx_{2}dt\longrightarrow \underset{0}{%
\overset{T}{\int }}\underset{P_{x_{3}}\Omega _{L}}{\int }%
u_{Li}u_{L1}D_{1}v_{i}dx_{1}dx_{2}dt
\end{equation*}%
because $u_{Li}^{m_{k}}\longrightarrow u_{Li}$ in $L^{2}\left( 0,T;H\right) $
and $u_{Li}^{m_{k}}\rightharpoonup u_{Li}$ in $L^{2}\left( 0,T;H\right) $
(at least) since $u_{L}^{m}$ belong, at least, to a bounded subset of $%
\mathcal{V}\left( Q_{L}^{T}\right) $ and (\ref{3.17}) holds for each therm.
Thus we obtain 
\begin{equation*}
\chi =B\left( u_{L}\right) \Longrightarrow B\left( u_{L}^{m_{k}}\right)
\rightharpoonup B\left( u_{L}\right) \quad \text{ in }L^{2}\left(
0,T;V^{\ast }\left( \Omega _{L}\right) \right)
\end{equation*}%
by using the density of $\mathcal{C}^{1}\left( \overline{Q}_{L}\right) $ in $%
\mathcal{V}\left( Q_{L}^{T}\right) $.
\end{proof}

Consequently, we proved the existence of the function $u_{L}\in \mathcal{V}%
\left( Q_{L}^{T}\right) $ that satisfies equation (\ref{3.6}) by applying to
this problem of the Faedo-Galerkin method by virtue of the above mentioned
results.

\subsection{\label{SS_5.5}Realisation of the initial condition}

The proof of the realisation of initial condition we can lead according to
same way as in \cite{Tem1} (see, also \cite{Lad1}, \cite{Lio1}), we will act
in just the same way.

Let $\phi $ be a continuously differentiable function on $[0,T]$ with $\phi
(T)=0$. We multiply (\ref{5.5}) by $\phi (t)$, and then integrate the first
term by parts. This leads to the equation 
\begin{equation*}
-\underset{0}{\overset{T}{\int }}\left\langle u_{Lm},\frac{d}{dt}\phi
(t)w_{j}\right\rangle _{\Omega _{L}}dt=\underset{0}{\overset{T}{\int }}%
\left\langle \nu \Delta u_{Lm},\phi (t)w_{j}\right\rangle _{\Omega _{L}}dt+
\end{equation*}%
\begin{equation*}
\underset{0}{\overset{T}{\int }}b\left( u_{Lm},u_{Lm},\phi (t)w_{j}\right)
dt+\underset{0}{\overset{T}{\int }}\left\langle f_{L},\phi
(t)w_{j}\right\rangle _{\Omega _{L}}dt+\left\langle u_{0Lm},\phi
(0)w_{j}\right\rangle _{\Omega _{L}}.
\end{equation*}

We can pass to the limit with respect to subsequence $\left\{
u_{Lm_{l}}\right\} _{l=1}^{\infty }$ of the sequence $\left\{ u_{Lm}\right\}
_{m=1}^{\infty }$ in all of terms by virtue of results which are proved in
above subsections. Then we find the equation 
\begin{equation*}
-\underset{0}{\overset{T}{\int }}\left\langle u_{L},\frac{d}{dt}\phi
(t)w_{j}\right\rangle _{\Omega _{L}}dt=\underset{0}{\overset{T}{\int }}%
\left\langle \nu \Delta u_{L},\phi (t)w_{j}\right\rangle _{\Omega _{L}}dt+
\end{equation*}%
\begin{equation}
\underset{0}{\overset{T}{\int }}b\left( u_{L},u_{L},\phi (t)w_{j}\right) dt+%
\underset{0}{\overset{T}{\int }}\left\langle f_{L},\phi
(t)w_{j}\right\rangle _{\Omega _{L}}dt+\left\langle u_{0L},\phi
(0)w_{j}\right\rangle _{\Omega _{L}},  \label{5.8}
\end{equation}%
that holds for each $w_{j}$, $j=1,2,...$. Consequently, this equation holds
for any finite linear combination of the $w_{j}$ and moreover because of a
continuity (\ref{5.8}) remains true for any $v\in V\left( \Omega _{L}\right) 
$.

Whence, one can draw conclusion that function $u_{L}$ satisfies equation (%
\ref{3.6}) in the distribution sense.

Now if we multiply (\ref{3.6}) by $\phi (t)$, and integrate with respect to $%
t$ after integrating the first term by parts, we get 
\begin{equation*}
-\underset{0}{\overset{T}{\int }}\left\langle u_{L},v\frac{d}{dt}\phi
(t)\right\rangle _{\Omega _{L}}dt-\underset{0}{\overset{T}{\int }}%
\left\langle \nu \Delta u_{L},\phi (t)v\right\rangle _{\Omega _{L}}dt+
\end{equation*}%
\begin{equation*}
\underset{0}{\overset{T}{\int }}\left\langle \underset{j=1}{\overset{3}{\sum 
}}u_{Lj}D_{j}u_{L},\phi (t)v\right\rangle _{\Omega _{L}}dt=\underset{0}{%
\overset{T}{\int }}\left\langle f_{L},\phi (t)v\right\rangle _{\Omega
_{L}}dt+\left\langle u_{L}\left( 0\right) ,\phi (0)v\right\rangle _{\Omega
_{L}}.
\end{equation*}

If we will compare this with (\ref{5.8}) after replacing $w_{j}$ with any $%
v\in V\left( \Omega _{L}\right) $ then we obtain 
\begin{equation*}
\phi (0)\left\langle u_{L}\left( 0\right) -u_{0L},v\right\rangle _{\Omega
_{L}}=0.
\end{equation*}%
Hence, we get the realisation of the initial condition by virtue of
arbitrariness of $v\in V\left( \Omega _{L}\right) $ and $\phi $, as one can
choose $\phi (0)\neq 0$.

Consequently, the following result is proved.

\begin{theorem}
\label{Th_2.1}Under above mentioned conditions for any 
\begin{equation*}
u_{0L}\in \left( H\left( \Omega _{L}\right) \right) ^{3},\quad f_{L}\in
L^{2}\left( 0,T;V^{\ast }\left( \Omega _{L}\right) \right)
\end{equation*}%
problem (\ref{3.3}) - (\ref{3.5}) has weak solution $u_{L}\left( t,x\right) $
that belongs to $\mathcal{V}\left( Q_{L}^{T}\right) $.
\end{theorem}

\begin{remark}
From the obtained a priori estimations and Propositions \ref{P_3.1} and \ref%
{P_3.2} follows of the fulfilment of all conditions of the general theorem
of the compactness method (see, e. g. \cite{Sol3}, \cite{SolAhm}, and for
complementary informations see, \cite{SolSpr}, \cite{Sol4}). We would like
to note also that the general theorems also were proved by using of the
Faedo-Galerkin method and $\varepsilon -$regularization.

Here we would like to note we could prove the existence of problem (\ref{3.3}%
) - (\ref{3.5}) by other way with using of the following general existence
theorem (see, \cite{Sol3}, \cite{SolAhm}) if in the adduced below theorem
the pn-space to replace onto $V\left( \Omega _{L}\right) $.

Let $X$ and $Y$ be Banach spaces with duals $X^{\ast }$ and $Y^{\ast }$
respectively, $Y$ be a reflexive Banach space, $\mathcal{M}_{0}\subseteq X$
be a weakly complete "reflexive" $pn-$space (see, Appendix A \cite{SolAhm}
or \cite{Sol3}), $X_{0}\subseteq \mathcal{M}_{0}\cap Y$ be a separable
vector topological space such that $\overline{X_{0}}^{\mathcal{M}_{0}}\equiv 
\mathcal{M}_{0}$, $\overline{X_{0}}^{Y}\equiv Y$. \ Consider the following
problem: 
\begin{equation}
\frac{dx}{dt}+f(t,x\left( t\right) )=y\left( t\right) ,\quad y\in
L^{p_{1}}\left( 0,T;Y\right) ;\quad x\left( 0\right) =0  \label{Eqn 2.4}
\end{equation}%
Let the following conditions be fulfilled: \newline
\textit{i}) $f:\underset{0}{\mathbf{P}}$\/$_{1,p_{0},p_{1}}\left( 0,T;%
\mathcal{M}_{0},Y\right) \rightarrow L^{p_{1}}\left( 0,T;Y\right) $ is a
weakly compact operator, where 
\begin{equation*}
\underset{0}{\mathbf{P}}\/_{1,p_{0},p_{1}}\left( 0,T;\mathcal{M}%
_{0},Y\right) \equiv L^{p_{0}}\left( 0,T;\mathcal{M}_{0}\right) \cap
W^{1,p_{1}}\left( 0,T;Y\right) \cap \left\{ x\left( t\right) \left\vert \
x\left( 0\right) =0\right. \right\} ,
\end{equation*}%
$1<\max \{p_{1},p_{1}^{`}\}\leq p_{0}<\infty $, $p_{1}^{\prime }=\frac{p_{1}%
}{p_{1}-1}$;\newline
(\textit{ii}) there is a linear continuous operator $L:W^{s,p_{2}}\left(
0,T;X_{0}\right) \rightarrow W^{s,p_{2}}\left( 0,T;Y^{\ast }\right) $, $%
s\geq 0$, $p_{2}\geq 1$ such that $L$ commutes with $\frac{d}{dt}$ and the
conjugate operator $L^{\ast }$ has $ker(L^{\ast })=\left\{ 0\right\} $;

\textit{(iii) }there exist a continuous function $\varphi :R_{+}^{1}\cup
\left\{ 0\right\} \longrightarrow R^{1}$ and numbers $\tau _{0}\geq 0$ and $%
\tau _{1}>0$ such that $\varphi (r)$ is nondecreasing for $\tau \geq \tau
_{0}$,\ $\varphi \left( \tau _{1}\right) >0$ and operators $f$ and $L$
satisfy the following inequality for any $x\in L^{p_{0}}\left(
0,T;X_{0}\right) $ 
\begin{equation*}
\underset{0}{\overset{T}{\int }}\langle f(t,x\left( t\right) ),Lx\left(
t\right) \rangle dt\geq \varphi \left( \lbrack x]_{L^{p_{0}}\left( \mathcal{M%
}_{0}\right) }\right) [x]_{L^{p_{0}}\left( \mathcal{M}_{0}\right) };
\end{equation*}%
\textit{(iv) }there exist a linear bounded operator $L_{0}:X_{0}\rightarrow
Y $ and constants $C_{0}>0$, $C_{1},C_{2}\geq 0$, $\nu >1$ such that the
inequalities 
\begin{eqnarray*}
\underset{0}{\overset{T}{\int }}\langle \xi \left( t\right) ,L\xi \left(
t\right) \rangle dt &\geq &C_{0}\left\Vert L_{0}\xi \right\Vert
_{L^{p_{1}}\left( 0,T;Y\right) }^{\nu }-C_{2}, \\
\underset{0}{\overset{t}{\int }}\langle \frac{dx}{d\tau },Lx\left( \tau
\right) \rangle d\tau &\geq &C_{1}\left\Vert L_{0}x\right\Vert _{Y}^{\nu
}\left( t\right) -C_{2},\quad a.e.\ t\in \left( 0,T\right]
\end{eqnarray*}%
hold for any $x\in W^{1,p_{0}}\left( 0,T;X_{0}\right) $ and $\xi \in
L^{p_{0}}\left( 0,T;X_{0}\right) $.
\end{remark}

\begin{theorem}
\label{Theorem 2.2}Assume that conditions (i) - (iv) are fulfilled. Then the
Cauchy problem (\ref{Eqn 2.4}) is solvable in $\underset{0}{\mathbf{P}}$\/$%
_{1,p_{0},p_{1}}\left( 0,T;\mathcal{M}_{0},Y\right) $ in the following sense 
\begin{equation*}
\underset{0}{\overset{T}{\int }}\left\langle \frac{dx}{dt}+f(t,x\left(
t\right) ),y^{\ast }\left( t\right) \right\rangle dt=\underset{0}{\overset{T}%
{\int }}\left\langle y\left( t\right) ,y^{\ast }\left( t\right)
\right\rangle dt,\quad \forall y^{\ast }\in L^{p_{1}\prime }\left(
0,T;Y^{\ast }\right) ,
\end{equation*}%
for any $y\in G\subseteq L^{p_{1}}\left( 0,T;Y\right) $, where $G\equiv $ $%
\underset{r\geq \tau _{1}}{\cup }G_{r}$:\newline
\begin{equation*}
G_{r}\equiv \left\{ y\in L^{p_{1}}\left( 0,T;Y\right) \left\vert \underset{0}%
{\overset{T}{\int }}\left\vert \langle y\left( t\right) ,Lx\left( t\right)
\rangle \right\vert ~dt\leq \underset{0}{\overset{T}{\int }}\langle
f(t,x\left( t\right) ),Lx\left( t\right) \rangle dt\right. -c,\right.
\end{equation*}%
\begin{equation*}
\left. \forall \text{ }x\in L^{p_{0}}\left( 0,T;X_{0}\right) ,\ \left[ x%
\right] _{L^{p_{0}}\left( 0,T;\mathcal{M}_{0}\right) }=r\right\} ,\
C_{2}<c<\infty .
\end{equation*}
\end{theorem}

The next proposition follows immediately from the theorem \ref{Theorem 2.2}.

\begin{corollary}
\label{Corollary 2.1}Under assumptions of Theorem \ref{Theorem 2.2} the
problem (\ref{Eqn 2.4}) is solvable in $\underset{0}{\mathbf{P}}$\/$%
_{1,p_{0},p_{1}}\left( 0,T;\mathcal{M}_{0},Y\right) $ for any $y\in
L^{p_{1}}\left( 0,T;Y\right) $ satisfying the condition: there is $r>0$ such
that the inequality 
\begin{equation*}
\left\Vert y\right\Vert _{L^{p_{1}}\left( 0,T;Y\right) }\leq \varphi \left(
\lbrack x]_{L^{p_{0}}\left( 0,T;\mathcal{M}_{0}\right) }\right)
\end{equation*}%
holds for any $x\in L^{p_{0}}\left( 0,T;X_{0}\right) $ with $%
[x]_{L^{p_{0}}\left( \mathcal{M}_{0}\right) }\geq r$. Furthermore, if $%
\varphi \left( \tau \right) \nearrow \infty $ as $\tau \nearrow \infty $
then the problem (\ref{Eqn 2.4}) is solvable in $\underset{0}{\mathbf{P}}$\/$%
_{1,p_{0},p_{1}}\left( 0,T;\mathcal{M}_{0},Y\right) $ for any $y\in
L^{p_{1}}\left( 0,T;Y\right) $ satisfying the inequality 
\begin{equation*}
\sup \left\{ \frac{1}{[x]_{L^{p_{0}}\left( 0,T;\mathcal{M}_{0}\right) }}%
\underset{0}{\overset{T}{\int }}\langle y\left( t\right) ,Lx\left( t\right)
\rangle ~dt\ \left\vert \ x\in L^{p_{0}}\left( 0,T;X_{0}\right) \right.
\right\} <\infty .
\end{equation*}
\end{corollary}

\section{\label{Sec_6}Uniqueness of Solution of Problem (3.3) - (3.5)}

For the study of the uniqueness of the solution as usually: we will assume
that the posed problem have two different solutions $u=\left(
u_{1},u_{2},u_{3}\right) $, $v=\left( v_{1},v_{2},v_{3}\right) $ and we will
investigate its difference: $w=u-v$.\ (Here for brevity we won't specify
indexes for functions, which shows that we investigate problem (\ref{3.3}) -
(\ref{3.5}) on $Q_{L}^{T}$.) Then for $w$ we obtain the following problem 
\begin{equation*}
\frac{\partial w}{\partial t}-\nu \left[ \left( 1+a_{1}^{-2}\right)
D_{1}^{2}+\left( 1+a_{2}^{-2}\right) D_{2}^{2}\right] w-2\nu
a_{1}^{-1}a_{2}^{-1}D_{1}D_{2}w+
\end{equation*}%
\begin{equation*}
\left( u_{1}-a_{1}^{-1}u_{3}\right) D_{1}u-\left(
v_{1}-a_{1}^{-1}v_{3}\right) D_{1}v+\left( u_{2}-a_{2}^{-1}u_{3}\right)
D_{2}u-
\end{equation*}%
\begin{equation*}
\left( v_{2}-a_{2}^{-1}v_{3}\right) D_{2}v=0,
\end{equation*}%
\begin{equation*}
\func{div}w=D_{1}\left[ \left( u-a_{1}^{-1}u_{3}\right) -\left(
v-a_{1}^{-1}v_{3}\right) \right] +D_{2}\left[ \left(
u-a_{2}^{-1}u_{3}\right) \right. -
\end{equation*}%
\begin{equation}
\left. \left( v-a_{2}^{-1}v_{3}\right) \right] =D_{1}w+D_{2}w-\left( a_{1}^{-1}D_{1}+a_{2}^{-1}D_{2}\right) w_{3}=0,
\label{3.18}
\end{equation}%
\begin{equation}
w\left( 0,x\right) =0,\quad x\in \Omega \cap L;\quad w\left\vert \ _{\left( 0,T\right) \times \partial \Omega _{L}}\right. =0.
\label{3.19}
\end{equation}

Hence, we derive 
\begin{equation*}
\frac{1}{2}\frac{d}{dt}\left\Vert w\right\Vert _{2}^{2}+\nu \left[ \left(
1+a_{1}^{-2}\right) \left\Vert D_{1}w\right\Vert _{2}^{2}+\left(
1+a_{2}^{-2}\right) \left\Vert D_{2}w\right\Vert _{2}^{2}\right] +
\end{equation*}%
\begin{equation*}
2\nu a_{1}^{-1}a_{2}^{-1}\left\langle D_{1}w,D_{2}w\right\rangle _{\Omega
_{L}}+\left\langle \left( u_{1}-a_{1}^{-1}u_{3}\right) D_{1}u-\left(
v_{1}-a_{1}^{-1}v_{3}\right) D_{1}v,w\right\rangle _{\Omega _{L}}+
\end{equation*}%
\begin{equation*}
\left\langle \left( u_{2}-a_{2}^{-1}u_{3}\right) D_{2}u-\left(
v_{2}-a_{2}^{-1}v_{3}\right) D_{2}v,w\right\rangle _{\Omega _{L}}=0
\end{equation*}%
or 
\begin{equation*}
\frac{1}{2}\frac{d}{dt}\left\Vert w\right\Vert _{2}^{2}+\nu \left(
\left\Vert D_{1}w\right\Vert _{2}^{2}+\left\Vert D_{2}w\right\Vert
_{2}^{2}\right) +\nu \left[ a_{1}^{-2}\left\Vert D_{1}w\right\Vert
_{2}^{2}+a_{2}^{-2}\left\Vert D_{2}w\right\Vert _{2}^{2}+\right.
\end{equation*}%
\begin{equation*}
\left. 2a_{1}^{-1}a_{2}^{-1}\left\langle D_{1}w,D_{2}w\right\rangle _{\Omega
_{L}}\right] +\left\langle u_{1}D_{1}u-v_{1}D_{1}v,w\right\rangle _{\Omega
_{L}}+\left\langle u_{2}D_{2}u-v_{2}D_{2}v,w\right\rangle _{\Omega _{L}}-
\end{equation*}%
\begin{equation}
a_{1}^{-1}\left\langle u_{3}D_{1}u-v_{3}D_{1}v,w\right\rangle _{\Omega
_{L}}-a_{2}^{-1}\left\langle u_{3}D_{2}u-v_{3}D_{2}v,w\right\rangle _{\Omega
_{L}}=0.  \label{3.20}
\end{equation}

If we consider the last 4 added elements of left part (\ref{3.20}),
separately, and if we simplify these by calculations then we get 
\begin{equation*}
\left\langle w_{1}D_{1}u,w\right\rangle _{\Omega _{L}}+\left\langle
v_{1}D_{1}w,w\right\rangle _{\Omega _{L}}+\left\langle
w_{2}D_{2}u,w\right\rangle _{\Omega _{L}}+\left\langle
v_{2}D_{2}w,w\right\rangle _{\Omega _{L}}-
\end{equation*}%
\begin{equation*}
a_{1}^{-1}\left\langle w_{3}D_{1}u,w\right\rangle _{\Omega
_{L}}-a_{1}^{-1}\left\langle v_{3}D_{1}w,w\right\rangle _{\Omega
_{L}}-a_{2}^{-1}\left\langle w_{3}D_{2}u,w\right\rangle _{\Omega
_{L}}-a_{2}^{-1}\left\langle v_{3}D_{2}w,w\right\rangle _{\Omega _{L}}=
\end{equation*}%
\begin{equation*}
\left\langle w_{1}D_{1}u,w\right\rangle _{\Omega _{L}}+\frac{1}{2}%
\left\langle v_{1},D_{1}w^{2}\right\rangle _{\Omega _{L}}+\left\langle
w_{2}D_{2}u,w\right\rangle _{\Omega _{L}}+\frac{1}{2}\left\langle
v_{2},D_{2}w^{2}\right\rangle _{\Omega _{L}}-
\end{equation*}%
\begin{equation*}
a_{1}^{-1}\left\langle w_{3}D_{1}u,w\right\rangle _{\Omega _{L}}-\frac{1}{2}%
a_{1}^{-1}\left\langle v_{3},D_{1}w^{2}\right\rangle _{\Omega
_{L}}-a_{2}^{-1}\left\langle w_{3}D_{2}u,w\right\rangle _{\Omega _{L}}-
\end{equation*}%
\begin{equation*}
\frac{1}{2}a_{2}^{-1}\left\langle v_{3},D_{2}w^{2}\right\rangle =\frac{1}{2}%
\left\langle v_{1}-a_{1}^{-1}v_{3},D_{1}w^{2}\right\rangle _{\Omega _{L}}+%
\frac{1}{2}\left\langle v_{2}-a_{2}^{-1}v_{3},D_{2}w^{2}\right\rangle
_{\Omega _{L}}+
\end{equation*}%
\begin{equation*}
\left\langle \left( w_{1}-a_{1}^{-1}w_{3}\right) w,D_{1}u\right\rangle
_{\Omega _{L}}+\left\langle \left( w_{2}-a_{2}^{-1}w_{3}\right)
w,D_{2}u\right\rangle _{\Omega _{L}}=
\end{equation*}%
\begin{equation*}
\left\langle \left( w_{1}-a_{1}^{-1}w_{3}\right) w,D_{1}u\right\rangle
_{\Omega _{L}}+\left\langle \left( w_{2}-a_{2}^{-1}w_{3}\right)
w,D_{2}u\right\rangle _{\Omega _{L}}.
\end{equation*}%
In the last equality we use the equation $\func{div}v=0$ (see, (\ref{3.4}))
and the condition (\ref{3.19}).

If we take into account this equality in equation (\ref{3.20}) then we get
the equation 
\begin{equation*}
\frac{1}{2}\frac{d}{dt}\left\Vert w\right\Vert _{2}^{2}+\nu \left(
\left\Vert D_{1}w\right\Vert _{2}^{2}+\left\Vert D_{2}w\right\Vert
_{2}^{2}\right) +\nu \left[ a_{1}^{-2}\left\Vert D_{1}w\right\Vert
_{2}^{2}+\right.
\end{equation*}%
\begin{equation*}
\left. a_{2}^{-2}\left\Vert D_{2}w\right\Vert
_{2}^{2}+2a_{1}^{-1}a_{2}^{-1}\left\langle D_{1}w,D_{2}w\right\rangle
_{\Omega _{L}}\right] +\left\langle \left( w_{1}-a_{1}^{-1}w_{3}\right)
w,D_{1}u\right\rangle _{\Omega _{L}}+
\end{equation*}%
\begin{equation}
\left\langle \left( w_{2}-a_{2}^{-1}w_{3}\right) w,D_{2}u\right\rangle
_{\Omega _{L}}=0,\quad \left( t,x\right) \in \left( 0,T\right) \times \Omega
_{L}  \label{3.21}
\end{equation}

Consequently, we derive the Cauchy problem for the equation (\ref{3.21})
with the initial condition 
\begin{equation}
\left\Vert w\right\Vert _{2}\left( 0\right) =0.  \label{3.22}
\end{equation}

Hence giving rise to the differential inequality we get the following Cauchy
problem for the differential inequality 
\begin{equation*}
\frac{1}{2}\frac{d}{dt}\left\Vert w\right\Vert _{2}^{2}+\nu \left(
\left\Vert D_{1}w\right\Vert _{2}^{2}+\left\Vert D_{2}w\right\Vert
_{2}^{2}\right) \leq
\end{equation*}%
\begin{equation}
\left\vert \left\langle \left( w_{1}-a_{1}^{-1}w_{3}\right)
w,D_{1}u\right\rangle _{\Omega _{L}}\right\vert +\left\vert \left\langle
\left( w_{2}-a_{2}^{-1}w_{3}\right) w,D_{2}u\right\rangle _{\Omega
_{L}}\right\vert ,  \label{3.21'}
\end{equation}

with the initial condition (\ref{3.22}).

We have the following estimate for the right side of (\ref{3.21'})%
\begin{equation*}
\left\vert \left\langle \left( w_{1}-a_{1}^{-1}w_{3}\right)
w,D_{1}u\right\rangle _{\Omega _{L}}\right\vert +\left\vert \left\langle
\left( w_{2}-a_{2}^{-1}w_{3}\right) w,D_{2}u\right\rangle _{\Omega
_{L}}\right\vert \leq
\end{equation*}%
\begin{equation*}
\left( \left\Vert w_{1}-a_{1}^{-1}w_{3}\right\Vert _{4}+\left\Vert
w_{2}-a_{2}^{-1}w_{3}\right\Vert _{4}\right) \left\Vert w\right\Vert
_{4}\left\Vert \nabla u\right\Vert _{2}\leq
\end{equation*}%
whence with the use of Gagliardo-Nirenberg-Sobolev inequality (\cite%
{BesIlNik}) we have 
\begin{equation*}
\left( 1+\max \left\{ \left\vert a_{1}^{-1}\right\vert ,\left\vert
a_{2}^{-1}\right\vert \right\} \right) \left\Vert w\right\Vert
_{4}^{2}\left\Vert \nabla u\right\Vert _{2}\leq c\left\Vert w\right\Vert
_{2}\left\Vert \nabla w\right\Vert _{2}\left\Vert \nabla u\right\Vert _{2}.
\end{equation*}%
It need to note that 
\begin{equation*}
\left( w_{1}-a_{1}^{-1}w_{3}\right) w,\ \left( w_{2}-a_{2}^{-1}w_{3}\right)
w\in L^{2}\left( 0,T;V^{\ast }\left( \Omega _{L}\right) \right) ,
\end{equation*}%
by virtue of (\ref{3.16}).

Now taking this into account in (\ref{3.21'}) one can arrive the following
Cauchy problem for differential inequality 
\begin{equation*}
\frac{1}{2}\frac{d}{dt}\left\Vert w\right\Vert _{2}^{2}\left( t\right) +\nu
\left\Vert \nabla w\right\Vert _{2}^{2}\left( t\right) \leq c\left\Vert
w\right\Vert _{2}\left( t\right) \left\Vert \nabla w\right\Vert _{2}\left(
t\right) \left\Vert \nabla u\right\Vert _{2}\left( t\right) \leq
\end{equation*}%
\begin{equation*}
C\left( c,\nu \right) \left\Vert \nabla u\right\Vert _{2}^{2}\left( t\right)
\left\Vert w\right\Vert _{2}^{2}\left( t\right) +\nu \left\Vert \nabla
w\right\Vert _{2}^{2}\left( t\right) ,\quad \left\Vert w\right\Vert
_{2}\left( 0\right) =0
\end{equation*}%
since $w\in L^{\infty }\left( 0,T;\left( H\left( \Omega _{L}\right) \right)
^{3}\right) $, and consequently, $\left\Vert w\right\Vert _{2}\left\Vert
\nabla w\right\Vert _{2}\in L^{2}\left( 0,T\right) $ by the virtue of the
above existence theorem $w\in \mathcal{V}\left( Q_{L}^{T}\right) $, here $%
C\left( c,\nu \right) >0$ is constant.

Thus we obtain the problem 
\begin{equation*}
\frac{d}{dt}\left\Vert w\right\Vert _{2}^{2}\left( t\right) \leq 2C\left(
c,\nu \right) \left\Vert \nabla u\right\Vert _{2}^{2}\left( t\right)
\left\Vert w\right\Vert _{2}^{2}\left( t\right) ,\quad \left\Vert
w\right\Vert _{2}\left( 0\right) =0,
\end{equation*}%
if we denote $\left\Vert w\right\Vert _{2}^{2}\left( t\right) \equiv y\left(
t\right) $ then 
\begin{equation*}
\frac{d}{dt}y\left( t\right) \leq 2C\left( c,\nu \right) \left\Vert \nabla
u\right\Vert _{2}^{2}\left( t\right) y\left( t\right) ,\quad y\left(
0\right) =0.
\end{equation*}

Consequently, we obtain $\left\Vert w\right\Vert _{2}^{2}\left( t\right)
\equiv y\left( t\right) =0$, i.e. the following result is proved: \ 

\begin{theorem}
\label{Th_2.2}Under above mentioned conditions for any 
\begin{equation*}
\left( f,u_{0}\right) \in L^{2}\left( 0,T;V^{\ast }\left( \Omega _{L}\right)
\right) \times \left( H\left( \Omega _{L}\right) \right) ^{3}
\end{equation*}%
problem (\ref{3.3}) - (\ref{3.5}) has a unique weak solution $u\left(
t,x\right) $ that is contained in $\mathcal{V}\left( Q_{L}^{T}\right) $.
\end{theorem}

\section{\label{Sec_7}Proof of Theorem \protect\ref{Th_1}}

\begin{proof}
(of Theorem \ref{Th_1}). As is known (\cite{Ler1}, \cite{Hop}, \cite{Lad2}, 
\cite{Lio1}), problem (1.1%
${{}^1}$%
) - (\ref{3}) is solvable and possesses weak solution that is contained in
the space $\mathcal{V}\left( Q^{T}\right) $. So, assume problem (1.1%
${{}^1}$%
) - (\ref{3}) has, at least, two different solutions under conditions of
Theorem \ref{Th_1}.

It is clear that if the problem have more than one solution then there is,
at least, some subdomain of $Q^{T}\equiv \left( 0,T\right) \times \Omega $,
on which this problem has, at least, two solutions such, that each from the
other are different. Consequently, starting from the above Lemma \ref{L_2.2}
we need to investigate the existence and uniqueness of the posed problem on
arbitrary fixed subdomain on which it is possibl that our problem can
possess more than one solution, more exactly in the case when the subdomain
is generated by an arbitrary fixed hyperplane by the virtue of Lemma \ref%
{L_2.2}. It is clear that, for us it is enough to prove that no such
subdomain generated by a hyperplane on which more than single solutions of
problem (1.1%
${{}^1}$%
) - (\ref{3}) exists, again by virtue of Lemma \ref{L_2.2}. In other words,
for us it remains to use the above results (i.e. Theorems \ref{Th_2.1} and %
\ref{Th_2.2}) in order to end the proof.

From the proved theorems above we obtain that there does not exist a
subdomain, defined in the previous section, on which problem (1.1%
${{}^1}$%
) - (\ref{3}) reduced on this subdomain might possesses more than one weak
solution. Consequently, taking Lemma \ref{L_2.2} into account we obtain that
the problem (1.1%
${{}^1}$%
) - (\ref{3}) (i.e.) under conditions of Theorem \ref{Th_1} possesses only
one weak solution.
\end{proof}

So, under conditions of Theorem \ref{Th_1} the uniqueness of the weak
solution $u(x,t)$ (of velocity vector) of the problem obtained from the
mixed problem for the incompressible Navier-Stokes equation by using
approach of the Hopf-Leray in three dimension case is proved as noted in
Notation 1.

Hence one can make the following conclusion

\section{Conclusion}

Let us 
\begin{equation*}
f\in L^{2}\left( 0,T;V^{\ast }\left( \Omega \right) \right) ,\ u_{0}\in
\left( H\left( \Omega \right) \right) ^{3}.
\end{equation*}%
It is well-known that the following inclusions are dense 
\begin{equation*}
L^{2}\left( 0,T;H^{1/2}\left( \Omega \right) \right) \subset L^{2}\left(
Q^{T}\right) ;\ \mathcal{H}^{1/2}\left( \Omega \right) \subset \left(
H\left( \Omega \right) \right) ^{3}\ \ \&
\end{equation*}%
\begin{equation*}
L^{2}\left( 0,T;\mathcal{H}^{1/2}\left( \Omega \right) \right) \subset
L^{2}\left( 0,T;\left( H^{-1}\left( \Omega \right) \right) ^{3}\right)
\end{equation*}%
consequently, there exist sequences 
\begin{equation*}
\left\{ u_{0m}\right\} _{m=1}^{\infty }\subset \mathcal{H}^{1/2}\left(
\Omega \right) ;\left\{ f_{m}\right\} _{m=1}^{\infty }\subset L^{2}\left(
0,T;\mathcal{H}^{1/2}\left( \Omega \right) \right)
\end{equation*}%
such that $u_{0m}\longrightarrow u_{0}$ in $\left( H\left( \Omega \right)
\right) ^{3}$ and $\left\Vert u_{0m}\right\Vert _{\left( H\left( \Omega
\right) \right) ^{3}}\leq \left\Vert u_{0}\right\Vert _{\left( H\left(
\Omega \right) \right) ^{3}}$, $f_{m}\longrightarrow f$ in $L^{2}\left(
0,T;\left( H^{-1}\left( \Omega \right) \right) ^{3}\right) $ and $\left\Vert
f_{m}\right\Vert _{L^{2}\left( 0,T;\left( H^{-1}\left( \Omega \right)
\right) ^{3}\right) }\leq \left\Vert f\right\Vert _{L^{2}\left( 0,T;\left(
H^{-1}\left( \Omega \right) \right) ^{3}\right) }$.

Consequently, for any $\varepsilon >0$ there exist $m\left( \varepsilon
\right) \geq 1$ such that under $m\geq m\left( \varepsilon \right) $ for the
corresponding elements $u_{0m}$, $f_{m}$ of the above sequences 
\begin{equation*}
\left\Vert u_{0}-u_{0m}\right\Vert _{\left( H\left( \Omega \right) \right)
^{3}}<\varepsilon ;\ \left\Vert f-f_{m}\right\Vert _{L^{2}\left( 0,T;\left(
H^{-1}\left( \Omega \right) \right) ^{3}\right) }<\varepsilon
\end{equation*}%
hold, and also the claim of Theorem \ref{Th_1} is valid for problem (1.1%
${{}^1}$%
) - (\ref{3}) with these datums.

One can note that the space that is everywhere dense subset of the necessary
space possesses the minimal smoothness with respect to the needed space and
also is sufficient for the application of our approach. So, we establish:

\begin{theorem}
\label{Th_8}Let $\Omega $ be a Lipschitz open bounded set in $R^{3}$. Then
the existing by Theorem 1 weak solution $u\in \mathcal{V}\left( Q^{T}\right) 
$ of the system (1.1%
${{}^1}$%
) - (\ref{3}) is unique, if the given functions $f$ and $u_{0}$\ satisfy of
conditions $f$ $\in L^{2}\left( 0,T;\mathcal{H}^{1/2}\left( \Omega \right)
\right) $ and $u_{0}\in \mathcal{H}^{1/2}\left( \Omega \right) $, where a
solution be understood in the sense of Definition \ref{D_2.2}, as is
well-known, spaces $L^{2}\left( 0,T;\mathcal{H}^{1/2}\left( \Omega \right)
\right) $ and $\mathcal{H}^{1/2}\left( \Omega \right) $ are everywhere dense
in spaces $L^{2}\left( 0,T;V^{\ast }\left( \Omega \right) \right) $ and $%
\left( H\left( \Omega \right) \right) ^{3}$, respectively.
\end{theorem}

\begin{acknowledgement}
The author express the gratitude to russian mathematicians of Moscow for
their giving useful suggestions and comments, which have helped to correct
and sufficiently much to improve this paper.
\end{acknowledgement}

\end{document}